\documentclass[11pt]{article}

\usepackage[margin=1in]{geometry}
\usepackage{amsmath,amssymb,amsthm}
\usepackage{mathtools}
\usepackage{graphicx}
\usepackage{microtype}
\usepackage{float}
\usepackage{algorithm}
\usepackage{algpseudocode}
\usepackage{placeins}
\usepackage{fvextra}
\usepackage{tikz}
\usetikzlibrary{positioning,arrows.meta}
\usepackage[hidelinks]{hyperref}
\usepackage[backend=biber,style=alphabetic,maxbibnames=99]{biblatex}
\theoremstyle{plain}
\newtheorem{theorem}{Theorem}[section]
\newtheorem{lemma}[theorem]{Lemma}

\theoremstyle{definition}
\newtheorem{definition}[theorem]{Definition}
\newtheorem{example}[theorem]{Example}

\theoremstyle{remark}

\title{\textbf{Another Brick in the Wall: A Computationally Assisted Proof of an Improved Upper Bound for Thrackles}}
\author{Truman Scoppetto}
\date{October 2026}

\begin{document}

\maketitle

\begin{abstract}
Let $G$ be a thrackle, where a thrackle is a graph drawing in which each edge meets every other edge exactly once, either at a vertex or at a crossing. It is conjectured that, for all thrackles, the number of edges is less than or equal to the number of vertices. Fulek and Pach created a program to determine whether or not a graph can be a thrackle. Through a reconstruction and improvement of their algorithm, we computationally prove a key dumbbell to be non-thrackleable. Using this exclusion, we improve Xu’s 8-face half-edge bound from five to four. Finally, using this stronger restriction with Xu's framework, together with separate arguments for graphs containing triangles and disconnected graphs, we improve the known upper bound toward this conjecture to $m\le1.385(n-1)$ for $n\ge4$.
\end{abstract}

%\tableofcontents

\section{Introduction}
When first looking at thrackles, one may notice that, in their most basic forms, the number of edges always remains less than or equal to the number of vertices. From this observation, John Conway posed the thrackle conjecture in the late 1960s~\cite{MiserehNikolayevsky2018Annular}, suggesting that a thrackle always has at least as many vertices as it has edges. A lay reader may understand only the words \emph{vertices} and \emph{edges}, thinking of how they compose shapes, but when those words are combined in this way with the elusive word ``thrackle,'' confusion may take hold. The problem, though, is quite simple to understand once one knows what a thrackle is. So, to fill in that blank, a thrackle is defined as a graph drawing---a construct composed of abstract nodes or points called vertices and segments that connect them called edges---in which every edge meets every other edge no more and no less than once. After understanding that, this problem becomes quite simple to understand.

Not so easy to understand, however, is how to prove or disprove this conjecture. Since the creation of this problem, many attempts have been made to solve it or get closer to an answer, and some of those attempts have produced significant updates to the upper bound. Prominent contributors have included Lov\'asz, Pach, Szegedy, Cairns, Nikolayevsky, Fulek, Goddyn, and Xu~\cite{LovaszPachSzegedy1997,CairnsNikolayevsky2000,FulekPach2011,GoddynXu2017,FulekPach2019,Xu2021}. Each has used different structural and computational strategies to advance both the understanding of this problem and the known bounds. Despite these successful attempts to get closer to an answer and the problem's long history, it remains open for anyone to try.

We will explore the approaches of some of these pioneers and how they obtained their improved bounds. Specifically, we will look at Xu's 2021 improvement and Fulek and Pach's 2011 improvement. We are looking at these two specifically because they both use unique strategies, ranging from structural to computational methods, that allow the problem to be seen in different lights and are directly used in the proof of our new bound. Additionally, we will go over my own exploration and improved bound, with the same coefficient that Fulek and Pach had previously observed conditionally~\cite[Remark~2]{FulekPach2019}, as well as my reconstruction of Fulek and Pach's thrackleability-testing program, which was not available to us.

As the reader, you are expected to have only a basic familiarity with graph theory; other than that, specialized concepts will be explained, and some basic algebra will be used. We will start with prerequisite information about concepts pertinent to this problem, such as basic definitions of a graph, foundational lemmas, and backbone equations. We will then move on to a deeper explanation of the problem, looking at why exactly it is hard and, in more depth, what it actually is. Following that, we will go over the previous attempts of Xu and of Fulek and Pach, as mentioned previously, with their discharging and computational/inequality methods, respectively. Finally, we will finish with my own exploration, which includes my attempts at a better bound, my replication and additions to Fulek and Pach's program that was not available to us, and future work.

We are currently working on developing this into a finalized, publishable paper without as much expository content and is more formally composed.

\begin{theorem}\label{thm:main-bound}
    A thrackleable graph \(G\) with \(n\) vertices and \(m\) edges satisfies
    \[
    m\le
\begin{cases}
n, & 1\le n\le3,\\
\frac{18}{13}(n-1), & n\ge4.
\end{cases}
\]
\end{theorem}
\section{Preliminaries}
\subsection{Basic Graph Theory Preliminaries}
A graph $G=(V,E)$ consists of a vertex set $V$ and an edge set $E$. The set $V$ is an abstract set of vertices (points or nodes), such as
\[
V=\{a,b,c\},
\] 
and $E$ is an abstract set of edges (vertex connectors), such as
\[
E=\bigl\{\{a,b\},\{b,c\},\{a,c\}\bigr\}.
\] 
A drawing of a graph is a visual representation that can have open crossings, meaning the touching of edges that do not share a vertex~\cite{FulekPach2011}, and consists of points and segments connecting them, as shown in Figure~\ref{fig:big-dipper}.
\begin{figure}[H]
    \centering
    \begin{tikzpicture}[scale=1.2,
        every node/.style={circle, fill=black, inner sep=2pt}]

        % Bowl
        \node[fill=red] (A) at (0,0) {};
        \node[fill=blue] (B) at (2,0.3) {};
        \node[fill=red] (C) at (2.4,1.8) {};
        \node[fill=blue] (D) at (0.4,1.5) {};

        % Handle
        \node[fill=blue] (E) at (3.6,2.3) {};
        \node[fill=red] (F) at (4.8,2.1) {};
        \node[fill=blue] (G) at (6.0,2.8) {};

        % Bowl edges
        \draw (A)--(B);
        \draw (B)--(C);
        \draw (C)--(D);
        \draw (D)--(A);

        % Handle edges
        \draw (C)--(E);
        \draw (E)--(F);
        \draw (F)--(G);

    \end{tikzpicture}

    \caption{A two-colored bipartite Big Dipper graph consisting of a $C_4$ cycle with a path attached to one vertex.}
    \label{fig:big-dipper}
\end{figure}
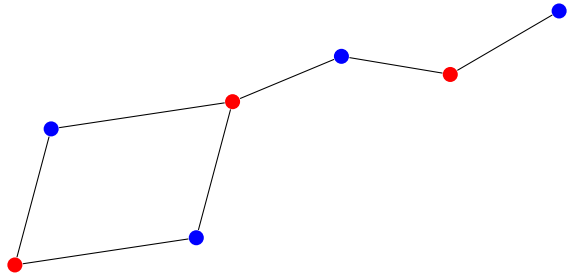

A graph $G$ is considered connected if there is a path between every pair of vertices. A graph is considered 2-connected if it has at least three vertices and remains connected after the removal of any one vertex. Equivalently, a 2-connected graph has no cut vertex, where a cut vertex is a vertex whose removal disconnects the graph. Throughout this paper, we define $|V(G)|$ as $n$ and $|E(G)|$ as $m$, and all graphs are finite and simple: they have finitely many vertices and edges, no loops, and at most one edge between any pair of vertices. 
\begin{definition}
    A cycle is a closed walk in which the first and last vertices are the same and no other vertices are repeated. This is denoted by $C_k$, where $k$ is the number of vertices in the cycle; for example, $C_4$ has a cycle length of four vertices.
\end{definition}
\begin{definition}
    A bipartite graph is one that contains no odd cycles ($C_3, C_5, C_7,\ldots$) and is two-colorable; that is, the vertices can be divided into two sets such that edges connect only vertices from different sets.
\end{definition}
\begin{definition}
    A plane embedding is a drawing of a graph in the plane with no open crossings. In an embedding, the plane is divided into connected regions called faces, and the boundary of a face is described by a boundary walk. A boundary walk is the closed walk formed by starting at a point on the boundary of a face and following the edges of the face until returning to the start.
\end{definition}
\begin{definition}
 The projective plane can be modeled by taking a disk and treating opposite points on the boundary as the same. A cellular embedding (also called a 2-cell embedding, meaning one composed of solid regions with no holes) in the projective plane is a representation of certain graphs in a more complex geometry that does not have a distinguished outer face.
\end{definition}

Euler's formula is an enduring relationship among faces, edges, vertices, and a constant called Euler's characteristic in graph theory. For a connected plane graph, Euler's formula is
\[
V-E+F=2,
\]
where $V$ is the number of vertices, $E$ is the number of edges, and $F$ counts the number of faces, including the outer face. For a cellular embedding of a connected graph in the projective plane, Euler's formula has an Euler characteristic of 1 and is instead
\[
V-E+F=1.
\]
Here, every face is counted, but unlike in the plane scenario, none of them is a designated outer face. This equation is one of the most important tools in graph theory. In our notation, the planar equation is
\[
n-m+f=2,
\]
and the projective-plane equation is $n-m+f=1$; these equations use the same quantities with different notation.

\begin{definition}
        In a graph $G$, a minor $H$ can be obtained by removing edges, removing vertices, or contracting edges. Contraction is the act of taking an edge with two endpoints, $e=uv$, and combining them into one vertex while keeping all edges incident to $u$ or $v$ attached to the new vertex. An important fact about minors is that if $H$ is a minor of $G$ and is nonplanar, then $G$ is nonplanar:
        \[
        H\preccurlyeq G \text{ and } H\text{ nonplanar}
        \quad\Longrightarrow\quad
        G\text{ nonplanar}.
        \]
\end{definition}
\subsection{General Thrackle Preliminaries}
A thrackle is a graph drawing in which every pair of edges meets exactly once, either at a common endpoint or at an interior proper crossing between nonincident edges which we call an open crossing. We will use the term open crossing rather than interior proper crossing throughout this paper. Important things to note about a thrackle drawing are that an edge cannot pass through a vertex that is not its endpoint and that no three edge interiors can pass through the same point. A graph is thrackleable if it admits a thrackle drawing. A thrackle drawing is demonstrated in Figure~\ref{fig:starcrossing}, with an open crossing marked as a white circle.
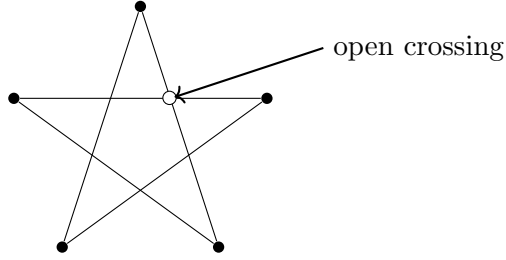
\begin{figure}[H]
    \centering
    \begin{tikzpicture}[scale=1.1, every node/.style={circle, fill=black, inner sep=1.5pt}]
        
        % Outer star points
        \node (A) at (90:1.6) {};
        \node (B) at (18:1.6) {};
        \node (C) at (-54:1.6) {};
        \node (D) at (-126:1.6) {};
        \node (E) at (162:1.6) {};
        
        % Star edges
        \draw (A) -- (C);
        \draw (C) -- (E);
        \draw (E) -- (B);
        \draw (B) -- (D);
        \draw (D) -- (A);
        
        % Open crossing point
        \draw[fill=white] (0.35,0.5) circle (0.08);
        
        % Arrow pointing to the open crossing
        \draw[->, thick] (2.2,1.1) -- (0.41,0.51);
        
        % Optional label
        \node[draw=none, fill=none, right] at (2.25,1.1) {open crossing};
    \end{tikzpicture}
    \caption{A star-shaped $C_5$ cycle thrackle with an indicated open crossing.}
    \label{fig:starcrossing}
\end{figure}
Additionally, thrackleability is hereditary with respect to subgraphs:
    \[
    H\subseteq G,\quad G\text{ thrackleable}
    \quad\Longrightarrow\quad
    H\text{ thrackleable}.
\]
Separately, a thrackleable graph cannot contain $C_4$ as a subgraph because $C_4$ is non-thrackleable~\cite{FulekPach2011}. Thus, as thrackleability is hereditary, any graph containing a $C_4$ is non-thrackleable.

\begin{definition}
    A dumbbell belongs to a family of graphs composed of two cycles. Dumbbells are denoted in the form
    \[
DB(c',c'',\ell),
\]
where $c'$ and $c''$ are positive integers greater than two representing the lengths of the two cycles, and $\ell$ is an integer satisfying $\ell>-\min(c',c'')$ that describes the method of connection between the cycles. When $\ell<0$, the cycles intersect exactly in a common path of $-\ell$ consecutive edges; when $\ell=0$, the cycles share exactly one vertex; and when $\ell>0$, the cycles are vertex-disjoint and have a path of $\ell$ edges connecting a vertex on one cycle to a vertex on the other, where none of the interior vertices on the path lie on either cycle~\cite{FulekPach2011}. Additionally, in a dumbbell,
\[
n=c'+c''+\ell-1,
\]
and
\[
m=c'+c''+\ell,
\]
and therefore
\[
m=n+1.
\]
This is illustrated in Figure~\ref{fig:db331-db330-db33m1}.
\end{definition}
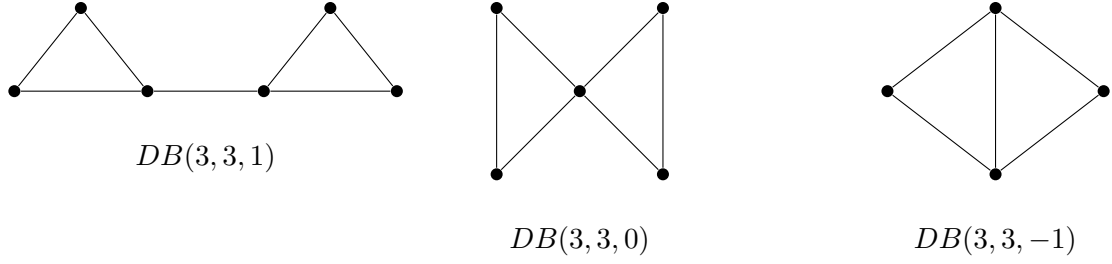
\begin{figure}[H]
\centering
\begin{tikzpicture}[scale=1.1, every node/.style={circle, fill=black, inner sep=1.6pt}]

    % ---------------- DB(3,3,1) ----------------
    \begin{scope}[xshift=0cm]
        % left triangle
        \node (A1) at (0,1) {};
        \node (B1) at (-0.8,0) {};
        \node (C1) at (0.8,0) {};
        \draw (A1)--(B1)--(C1)--(A1);

        % right triangle
        \node (D1) at (3,1) {};
        \node (E1) at (2.2,0) {};
        \node (F1) at (3.8,0) {};
        \draw (D1)--(E1)--(F1)--(D1);

        % connecting path of length 1
        \draw (C1)--(E1);

        % label
        \node[draw=none, fill=none] at (1.5,-0.8) {$DB(3,3,1)$};
    \end{scope}

    % ---------------- DB(3,3,0) ----------------
    \begin{scope}[xshift=6cm]
        % shared vertex
        \node (V) at (0,0) {};

        % left triangle
        \node (A2) at (-1,1) {};
        \node (B2) at (-1,-1) {};
        \draw (V)--(A2)--(B2)--(V);

        % right triangle
        \node (C2) at (1,1) {};
        \node (D2) at (1,-1) {};
        \draw (V)--(C2)--(D2)--(V);

        % label
        \node[draw=none, fill=none] at (0,-1.8) {$DB(3,3,0)$};
    \end{scope}

    % ---------------- DB(3,3,-1) ----------------
    \begin{scope}[xshift=11cm]
        % shared edge
        \node (A3) at (0,1) {};
        \node (B3) at (0,-1) {};
        \node (C3) at (-1.3,0) {};
        \node (D3) at (1.3,0) {};

        % two triangles sharing one edge
        \draw (A3)--(C3)--(B3)--(A3);
        \draw (A3)--(D3)--(B3);

        % label
        \node[draw=none, fill=none] at (0,-1.8) {$DB(3,3,-1)$};
    \end{scope}

\end{tikzpicture}
\caption{The dumbbell graphs $DB(3,3,1)$, $DB(3,3,0)$, and $DB(3,3,-1)$.}
\label{fig:db331-db330-db33m1}
\end{figure}
\section{Known Results/Major Approaches}
\subsection{Xu's 2021 Proof}
Xu achieved the most recent bound of $m\le1.393(n-1)$. This was done mainly through the discharging method and his innovative new discharging rule. While it is easy just to mention the existence of discharging and present the proof, an understanding of the discharging method is crucial to fully comprehending Xu's proof. Thus, we will describe discharging briefly before continuing.

Discharging is a proof method used in graph theory that takes inspiration from the way charges work in physics. These charges are allotted to all faces and vertices of a graph and are determined by a charge formula. When discharging, one must know the initial charge before continuing. A common way to do this is to design discharging equations around Euler's formula, as Xu does. After establishing the charges, one creates a discharging rule. The discharging rule describes how charges will be moved and reallocated within a graph. The final goal of discharging is for the rules to result in local configurations whose charges force a contradiction with the known global charge. We will provide some quick facts followed by an example to illustrate this method before presenting Xu's proof.
\begin{definition}
    A vertex degree is denoted by $d(v)$, where $d(v)$ is equal to the number of edges incident to a vertex $v$. For example, if $v$ is incident to six edges, then $d(v)=6$. Likewise, a face degree is denoted by $d(f)$, where $d(f)$ is equal to the length of a boundary walk on $f$, counted with multiplicity.
\end{definition}
\begin{lemma}
The sum of all vertex degrees is
\[
\sum_{v \in V} d(v) = 2|E|.
\]
Likewise, the sum of all face degrees is
\[
\sum_{f \in F} d(f) = 2|E|.
\]
This is called the Handshaking Lemma.
\end{lemma}
\begin{proof}
Each edge $e=uv$ has two vertex endpoints. Because a degree counts all incident edges, each endpoint of each edge counts it once. Thus, each edge is counted twice, resulting in
\[
\sum_{v \in V} d(v) = 2|E|.
\]
For faces, the idea is the same. Every edge has two sides, and each side contributes one boundary incidence, one on each side. Because each face counts all edges around it, each edge is counted twice, resulting in
\[
\sum_{f \in F} d(f) = 2|E|.
\]
\end{proof}
\begin{example}
    We start with the claim that a connected plane graph with minimum degree at least 3 has a face of degree at most 5. We now begin the proof.
    \begin{proof}
        For a contradiction, assume the opposite:
        \[
d(v) \ge 3 \qquad \text{for every vertex } v
\]
and
\[
d(f) \ge 6 \qquad \text{for every face } f.
\]
        We also assign the charges to be
        \[
\operatorname{ch}_0(v) = d(v) - 4,
\qquad
\operatorname{ch}_0(f) = d(f) - 4.
\]
This is useful because we can immediately see that the total charge is
\[
\begin{aligned}
\operatorname{ch}
&=
\sum_{v \in V} \operatorname{ch}_0(v)
+
\sum_{f \in F} \operatorname{ch}_0(f)
\\
&=
2m - 4|V| + 2m - 4|F|
\\
&=
-4\bigl(|V| + |F| - m\bigr)
\\
&=
-8.
\end{aligned}
\]
Finally, we adopt our discharging rule: every face sends $\frac{1}{3}$ unit of charge to each occurrence of a vertex on its boundary walk.
Now, a degree-3 vertex begins with
\[
3 - 4 = -1.
\]
After following the discharging rule, though, it receives $\frac{1}{3}$ unit of charge for each of its incident corners:
\[
3\left(\frac{1}{3}\right)=1,
\]
meaning that its charge becomes
\[
3\left(\frac{1}{3}\right)-1=0.
\]
Next, degree-$4^+$ vertices already start with a charge of at least
\[
d(v)-4\ge0,
\]
since 
\[
4 - 4 = 0.
\]
Finally, faces have a starting charge of
\[
d(f)-4,
\]
which, after discharging $\frac{1}{3}$ through each boundary incidence, becomes
\[
\frac{2d(f)}{3}-4,
\]
and, since $d(f)\ge6$,
\[
\frac{2d(f)}{3} - 4 \ge 0.
\]
Thus, because all faces and vertices have nonnegative final charge,
\[
\sum \operatorname{ch}^*\ge0.
\]
However, we only moved charge instead of creating it, and
\[
\sum \operatorname{ch}=-8,
\]
which is a contradiction. Thus, $d(f)\le5$ for some face $f$.

    \end{proof}
\end{example}
Now we can move on to explaining Xu's paper.

\begin{figure}[htbp]
\centering
\begin{tikzpicture}[scale=0.85]

% Central 8-face
\draw
(0,2) -- (1,2) -- (2,1) -- (2,0) --
(1,-1) -- (0,-1) -- (-1,0) -- (-1,1) -- cycle;

% Upper-left 6-face
% Shares TWO edges of the 8-face:
% (-1,1)--(0,2) and (0,2)--(1,2)
\draw
(-1,1) -- (-1.5,2) -- (-0.5,3) -- (1,3) -- (1,2);

% Lower-left 6-face
% Shares TWO edges of the 8-face:
% (-1,0)--(0,-1) and (0,-1)--(1,-1)
\draw
(-1,0) -- (-1.5,-1) -- (-0.5,-2) -- (1,-2) -- (1,-1);

% Right 6-face
% Shares ONE edge of the 8-face:
% (2,1)--(2,0)
\draw
(2,1) -- (2.5,1.5) -- (3,1) -- (3,0)
-- (2.5,-0.5) -- (2,0);

% Vertices
\foreach \x/\y in {
0/2, 1/2, 2/1, 2/0, 1/-1, 0/-1, -1/0, -1/1,
-1.5/2, -0.5/3, 1/3,
-1.5/-1, -0.5/-2, 1/-2,
2.5/1.5, 3/1, 3/0, 2.5/-0.5
}{
    \fill (\x,\y) circle (2pt);
}

\end{tikzpicture}
\caption{An \(8\)-face with five half-edges, adapted from Xu \cite{Xu2021}.}
\label{fig:xu-five-half-edges}
\end{figure}
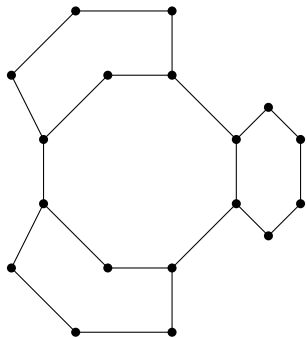
\begin{definition}
    A half-edge is an edge shared by both a 6-face and an $8^+$-face. For a face $F$, the number of half-edges incident to $F$ is denoted by $h(F)$~\cite{Xu2021}.
\end{definition}
\begin{lemma}[Xu, Lemma~1.1~\cite{Xu2021}]
    Let $G$ be a connected thrackleable graph. If $G$ is bipartite, it has a plane embedding; if it is nonbipartite, it has a 2-cell embedding in the projective plane in which every face is even.
\end{lemma}

This is important when it comes to how Xu represents his graph in the proof: not as a thrackle drawing, but as a plane or projective-plane embedding, as appropriate.

\begin{lemma}[Xu, Lemma~2.1~\cite{Xu2021}]
    There are at most five half-edges on the boundary of each 8-face.
\end{lemma}
\begin{proof}
    Previous work had established that there are at most six half-edges on the boundary of each 8-face~\cite[Lemma~2.5]{GoddynXu2017}. Xu eliminates every six-half-edge scenario, establishing that there cannot be a thrackleable scenario with six half-edges on an $8$-face. Thus, there are at most five half-edges on an 8-face.
\end{proof}
\begin{theorem}[Xu, Theorem~1.1~\cite{Xu2021}]
     A thrackleable graph of order $n$ has at most $1.393(n-1)$ edges.
\end{theorem}
\begin{proof}
    Using a five-half-edge bound, Xu constructs a discharging proof. He first defines the charges such that they satisfy 
    \[
    r+\sum_{v\in V} t(v)+\sum_{f\in\mathcal{F}} t(f) = -(k-1)r,
    \]
    where $k$ is the Euler characteristic, with $k=2$ for the plane and $k=1$ for the projective plane. He does this so that we know the total charge, which he finds to be nonpositive. Because the proof is designed in this way, his assumption becomes $m>c(n-1)$, and the total-charge calculation also gives the form
    \[
    m>\frac{r}{2}(n-1),
    \]
    so he sets $r=2c$. Now that this preliminary part is done, he begins to solve for $c$ using the following discharging rule: all $8^+$-faces evenly distribute charge through each of their half-edges to adjacent $6$-faces. He chooses $1.393$ as the value of $c$ to keep the face charge nonnegative and shows that the initial assumption ensures
    \[
    m>\frac r2(n-1)\Rightarrow r+\sum_{v\in V} t(v)>0.
    \]
    Thus, because \(\sum_f t^*(f)\ge0\) and \(r+\sum_v t(v)>0\), the total charge is positive. Discharging preserves total charge, and because the initial total charge was \(-(k-1)r\le0\), the nonpositive initial charge contradicts the positive final charge~\cite[Section~3]{Xu2021}.
\end{proof}
    This bound does not hold for $C_3$, so in this paper it is used only for $n\ge4$.
\subsection{Fulek and Pach's 2011 Proof}
Fulek and Pach's 2011 proof combines both computational and structural approaches to achieve its bound. By establishing the dumbbell and a method for showing non-thrackleability, this paper could be considered a backbone of the most modern approaches to the conjecture. For this reason, while we will go over how they obtained their bound through their structural approach, we will place a larger emphasis on their thrackleability-testing algorithm.

\begin{theorem}[Fulek--Pach, Theorem~2~\cite{FulekPach2011}]
The maximum number of edges in a thrackle on \(n\) vertices satisfies
\[
t(n)\le \frac{167}{117}n<1.428n.
\]
\end{theorem}
This is the result that Fulek and Pach obtained by excluding the dumbbells $DB(6,6,-2)$, $DB(6,6,-1)$, and $DB(6,6,0)$~\cite[Lemma~6]{FulekPach2011} and using the earlier exclusion of $DB(6,6,-3)$~\cite[Theorem~5.1]{LovaszPachSzegedy1997}. They were thus able to apply $c=6$ and $\ell=0$ in their inequality to obtain the bound.
\begin{definition}
    If $E_e$ is the set of edges nonincident to $e$, then $e_x\in E_e$ if $e_x$ does not share a vertex with edge $e$.
\end{definition}
\begin{definition}
    For each edge $e$, assign an arbitrary direction. The crossing order of edges in $E_e$ along an edge $e$, read in that arbitrary direction, is denoted by $\pi_e$. A complete row $\pi_e$ contains every edge in $E_e$ exactly once, while a partial row contains only the crossings of $e$ fixed by the search.
\end{definition}
\begin{definition}
    The set of all edge-crossing orders is denoted by $\Pi=(\pi_e)_{e\in E(G)}$.
\end{definition}
\begin{definition}
    The number of edges in the set of nonincident edges is denoted by $m(e)=|E_e|$.
\end{definition}
These crossing-order definitions follow Fulek and Pach's definitions~\cite[Section~4]{FulekPach2011}.

To test thrackleability, their algorithm takes an abstract graph $G=(V,E)$, with $V$ as the set of vertices and $E$ as the set of edges. They next construct an ordering $\pi_e$ of all nonincident edges according to where they would cross a given edge $e$, repeating this for each edge until they have a set of these orderings, $\Pi=\{\pi_e\mid e\in E(G)\}$. Following this, they replace each abstract open crossing with a vertex and subdivide the two edges belonging to the open crossing at the site of the new vertex. After this happens, if a segment of an original edge directly connects two crossing vertices, then an additional vertex is added on this segment, subdividing it once again. For the last step here, they construct a diamond shape around each open-crossing vertex so that the vertices of the diamond are the four adjacent vertices on the original edges, forming a new graph $G'$. This diamond serves as a sort of lock when running a planarity program, ensuring that the local ordering of edges around the open-crossing vertex remains the same. To see why this is important, suppose that, when walking clockwise around a crossing and recording the order in which the edges appear, we obtain $e_1e_2e_3e_4$, where $e_1$ and $e_3$ originally belonged to the same edge before it was split into two, and $e_2$ and $e_4$ originally belonged to the other edge.
\begin{figure}[H]
\centering
\begin{tikzpicture}[scale=1.2,
    vertex/.style={circle, fill=black, inner sep=1.5pt},
    every node/.style={font=\small}
]

% Center

\node[vertex] (O) at (0,0) {};

% Outer endpoints of plus sign
\coordinate (N) at (0,2);
\coordinate (E) at (2,0);
\coordinate (S) at (0,-2);
\coordinate (W) at (-2,0);

% Inner vertices on each arm
\node[vertex] (n) at (0,1) {};
\node[vertex] (e) at (1,0) {};
\node[vertex] (s) at (0,-1) {};
\node[vertex] (w) at (-1,0) {};

% Draw plus sign edges
\draw[thick] (N)--(n)--(O)--(s)--(S);
\draw[thick] (W)--(w)--(O)--(e)--(E);

% Diamond around the center
\draw[thick] (n)--(e)--(s)--(w)--(n);

% Labels for the enclosed edge-segments
\node at (0.2,0.45) {$e_1$};   % north
\node at (0.45,-0.15) {$e_2$};  % east
\node at (-0.2,-0.45) {$e_3$}; % south
\node at (-0.45,0.15) {$e_4$};  % west

\end{tikzpicture}
\caption{An open crossing vertex surrounded by its diamond configuration lock.}
\end{figure}
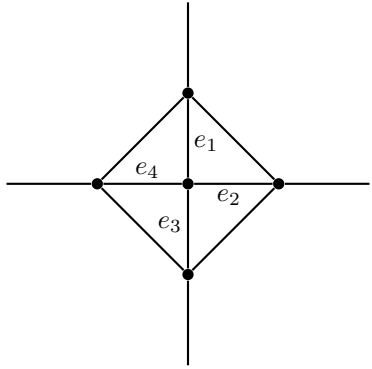
For a crossing to be an open crossing, its edge components must be opposite one another. If the planarity test were run without the diamond lock, however, the clockwise edge ordering might look like $e_1e_4e_2e_3$. This would place components belonging to the same original edge next to each other, so the vertex would not represent an open crossing. The graph $G$ is thrackleable if and only if there exists a complete $\Pi$ such that $G'(G,\Pi)$ is planar. Thus, after constructing a complete $G'$, a standard planarity test is run to see whether it is planar. If it is, then the graph $G$ is known to have a thrackleable configuration under the ordered-crossing system $\Pi$. This is because when $G'$ is made planar, it has no open crossings, but the open crossings of $G$ are already represented by vertices locked in place. Changing these vertices back into open crossings makes them the only open crossings in the graph, meaning that it is thrackleable~\cite[Section~4]{FulekPach2011}.

Fulek and Pach also construct a backtracking algorithm to find the proper $\Pi$. At a high level, the ordered-crossing sets begin empty. Additionally, for a partial crossing-order system $\Pi''$, let \(G'_{\mathrm{pre}}(G,\Pi'')\) denote its augmented graph constructed along a fixed Euler traversal. This graph contains only portions of the active edge that have been traced. Thus, if the current edge is unfinished, the segment from its last recorded crossing to its terminal vertex is omitted. This also means that, for every crossing vertex, only components of the locking configuration supported by existing arms are present. Finally, for every complete \(\Pi\) extending \(\Pi''\), the graph \(G'_{\mathrm{pre}}(G,\Pi'')\) is a minor of \(G'(G,\Pi)\). Therefore, if \(G'_{\mathrm{pre}}(G,\Pi'')\) is nonplanar, every completion is nonplanar; if it is planar, the branch has simply not yet been ruled out.

The program first checks whether every ordered open-crossing set in $\Pi$ is filled with every edge nonincident to the edge to which the set belongs, returning true if this is the case. This is because being filled means that the configuration that reached this point must have a complete $\Pi$ that successfully passed the planarity test, so it is thrackleable. If this check returns false, the program then checks whether $e=-1$. If it does, the program picks a new edge to work on because $-1$ means that the previous edge being worked on has filled out its ordered crossings and there is now no current edge. This is safe because the most recently tested \(G'_{\mathrm{pre}}(G,\Pi'')\) is planar, so the branch has not yet been ruled out. The next section of the code can set $e=-1$: it checks whether an edge has filled out its ordered crossings and, if so, calls the function recursively with $e=-1$.

If an edge has not filled out its ordered crossings, the program cycles through each nonincident edge $e'\in E_e$ that the current edge $e$ has yet to cross. In this process, the program tries to place the current edge $e$ in every possible position in the ordered-crossing set of each currently observed nonincident edge $e'$, while the current nonincident edge $e'$ is appended to the end of the ordered-crossing set of the current edge $e$. Each resulting order and insertion choice receives a planarity check. If \(G'_{\mathrm{pre}}(G,\Pi'')\) is planar, the program calls itself recursively. If that call returns false, then, although the current configuration worked, all deeper permutations with additional edges failed, so the decision about where the current edge and observed nonincident edge are placed in the ordered-crossing sets is removed. If \(G'_{\mathrm{pre}}(G,\Pi'')\) is nonplanar instead, then the decision about where the edges are placed in the ordered-crossing sets is removed, and the program cycles to a new crossing-order configuration for $e$ and $e'$ or selects a new nonincident edge $e'$ whose crossing-order permutations will be checked. This is because, for a partial set of edge-crossing orders $\Pi''$, when a planarity test finds \(G'_{\mathrm{pre}}(G,\Pi'')\) to be nonplanar, the crossing-order sets that could result from adding an edge to this set are pruned, so a new crossing-order choice is needed to continue. Finally, if none of the positions for any of the nonincident edges work, then this backtrack fails and the program returns to the preceding backtracking step. The number of possible complete ordered-crossing systems is given by
\[
\prod_{e\in E(G)} m(e)!.
\]
The actual pseudocode is as follows:
\begin{figure}[H]
\centering
\begin{minipage}{0.82\textwidth}

\begin{tabbing}
\hspace{1em}\=\hspace{1em}\=\hspace{1em}\=
\hspace{1em}\=\hspace{1em}\=\kill

$\textsc{Backtrack}(e,\Pi'):$\\

\> if $\Pi'$ is complete\\
\>\> return True\\

\> if $e=-1$\\
\>\> $e\leftarrow\textsc{PickNextEdge}(G, \Pi')$\\[1mm]

\> if all required crossings of $e$ have been inserted\\
\>\> return $\textsc{Backtrack}(-1,\Pi')$\\[1mm]

\> for each nonincident edge $e'$ not yet crossed by $e$\\
\>\> for each possible position $p$ in $\pi_{e'}$\\
\>\>\> $\Pi''\leftarrow
        \textsc{Update}(\Pi',e,e',p)$\\

\>\>\> if \(G'_{\mathrm{pre}}(G,\Pi'')\) is planar\\
\>\>\>\> if $\Pi''$ is complete\\
\>\>\>\>\> return \textsc{True}\\
\>\>\>\> if $\textsc{Backtrack}(e,\Pi'')=\textsc{True}$\\
\>\>\>\>\> return \textsc{True}\\[1mm]

\> return \textsc{False}

\end{tabbing}

\end{minipage}
\caption{High-level reformulation of the Fulek--Pach backtracking
algorithm~\cite[Appendix~A]{FulekPach2011}.}
\label{alg:fulek-pach}
\end{figure}
\begin{lemma}
    The dumbbell graphs $DB(6,8,\ell)$ are non-thrackleable for $-5\le\ell\le-3$.
\end{lemma}
\begin{proof}
    Using their algorithm, Fulek and Pach~\cite[Introduction and Section~4]{FulekPach2011} verified Conway's conjecture for bipartite graphs on at most $10$ vertices. For dumbbells of the form $DB(6,8,\ell)$,
\[
n=8+6+\ell-1 = 13+\ell,\qquad m=8+6+\ell=n+1.
\]
For $\ell=-5,-4,-3$, the corresponding values are $n=8,9,10$. These graphs are bipartite, so if any of them were thrackleable, there would be a bipartite graph with at most 10 vertices satisfying $m=n+1$, contradicting $m\le n$.
\end{proof}
\section{Computational Exclusion of \texorpdfstring{$DB(6,8,-2)$}{DB(6,8,-2)}}
To rule out the $DB(6,8,-2)$ dumbbell, we adapted Fulek and Pach's program to be more efficient specifically for this graph. To do this, we employed multiple strategies that all revolve around Lemma~\ref{lem:larger-system-exclusions}.
\begin{lemma}\label{lem:larger-system-exclusions}
    For a graph $G$, let $S\subset E(G)$ denote a subset of the original edges of $G$. Let $G_S$ be the graph $G$ with the edges outside $S$ removed, and let $\Pi_S$ be a complete ordered-crossing system for $G_S$. Now let $G'(G_S,\Pi_S)$ be the Fulek--Pach augmented graph, including open-crossing vertices, subdivided edges, and lock structures, formed from the edges of $G$ that belong to $S$ and the crossings among them. Finally, let $\operatorname{Ext}(\Pi_S)$ be the set of complete configurations $\Pi$ whose restriction to the edges in $S$ is exactly $\Pi_S$. If $G'(G_S,\Pi_S)$ is nonplanar, then $G'(G,\Pi)$ is nonplanar for all $\Pi\in\operatorname{Ext}(\Pi_S)$.
\end{lemma}
\begin{proof}
    For an arbitrary $\Pi\in\operatorname{Ext}(\Pi_S)$, remove all edges outside the subset $S$ from $G$, along with all augmented structures defined in Fulek and Pach's algorithm that belong to original edges outside $S$, and contract the unnecessary subdivision paths left along the retained edges. The resulting graph is a minor because it is obtained from $G'(G,\Pi)$ by removing edges and contracting edges or paths, which are the key operations that define a minor. Additionally, because $\Pi\in\operatorname{Ext}(\Pi_S)$, deleting crossings involving edges outside $S$ leaves the crossing order of the edges in $S$ unchanged. Thus,
    \[
    G'(G_S,\Pi_S)\preccurlyeq G'(G,\Pi)
    \qquad\text{for every }\Pi\in\operatorname{Ext}(\Pi_S).
    \]
    By the definition of a minor given in the preliminaries, if $G'(G_S,\Pi_S)$ is nonplanar, then $G'(G,\Pi)$ is nonplanar for all $\Pi\in\operatorname{Ext}(\Pi_S)$.
\end{proof}
This differs from the chronological-prefix construction above: \(\Pi_S\) is complete on the selected graph \(G_S\), whereas \(\Pi''\) in \(G'_{\mathrm{pre}}(G,\Pi'')\) is incomplete and is tested using the fixed chronological-prefix construction. 

We incorporated this idea into our thrackleability-testing algorithm so that it constantly checks different choices of $S$. Completely testing each choice does not just eliminate one $\Pi$; it eliminates every $\Pi$ that can be made by extending $\Pi_S$, thus creating a speedup.

A key part of our algorithm that implements this lemma dictates how to choose $S$. We begin by defining the edges in the subset $S$ to be every edge in $C_6$. Under this choice, there are
\[
\prod_{e\in E(C_6)} 3! = (3!)^6 = 46,\!656
\]
permutations of $\Pi_S$ to check. Of these, eight permutations are planar when checked. We continue to check more cases using these eight cores as building blocks. This is done by adding one edge from the $C_8$ to one of the cores and checking every possible permutation. The new edge is tried in every position in the already established crossing order of each $C_6$ edge, and every crossing order for the new edge is also tried. With $x$ as a $C_8$ edge and $a$ through $f$ as the $C_6$ edges nonincident to $x$, an example, but not a full enumeration, looks as follows:
\begin{center}
\begin{tabular}{c|c}
\hline
Permutation & Possible orders \\
\hline
Permutation 1 $(\pi_a,\pi_x)$ &
$\pi_a=[x,c,d,e],\quad \pi_x=[a,b,c,d,e,f]$ \\[4pt]

Permutation 2 $(\pi_a,\pi_x)$ &
$\pi_a=[c,x,d,e],\quad \pi_x=[b,a,c,d,e,f]$ \\[4pt]

Permutation 3 $(\pi_a,\pi_x)$ &
$\pi_a=[c,d,x,e],\quad \pi_x=[c,a,b,d,e,f]$ \\[4pt]

Permutation 4 $(\pi_a,\pi_x)$ &
$\pi_a=[c,d,e,x],\quad \pi_x=[f,e,d,c,b,a]$ \\
\hline
\end{tabular}
\end{center}
When adding a new edge $x$, the program must decide two things to ensure that every possible ordered-crossing permutation including $x$ is tested. The first decision is where $x$ will be placed in the already established ordered-crossing sets $\pi_e$ of the previous nonincident edges. The second decision is the order of the nonincident crossings in $x$'s own ordered-crossing set $\pi_x$. In this way, the program cycles through every possible combination of the two. This ensures that no safe branch is missed. If the edges following $x$ are removed from all ordered-crossing sets $\pi_e$, and their own ordered-crossing sets are also removed, then any complete planar $\Pi$ produces one of the configurations involving $x$ that were searched here. Thus, every successful $\Pi$ must contain a tested configuration involving $x$. Because planarity cannot be regained by adding edges to a nonplanar configuration, pruning the branches following a nonplanar $x$ permutation removes only impossible branches.

Additionally, if $x$ were added after a different ordering of edges, removing those added edges would leave an already tested core-plus-$x$ system. Suppose that we add $x$ immediately after the core and obtain $z$ possible planar permutations. Now suppose that $x$ is added after a different edge, making $x$ the second added edge. If we remove that new first edge from all ordered-crossing sets and remove its own ordered-crossing set, as we did for the edges following $x$ earlier, we again obtain the core and $x$. The edge that preceded $x$ in this scenario can therefore be viewed as an extension of one of the core-plus-$x$ permutations, because adding its crossings back to recover the original permutation simply extends the core and $x$. Because planarity cannot be regained from nonplanarity, the resulting core-plus-$x$ system must be one of the $z$ successful orderings. This means that only one fixed edge-addition order is needed to cover all possible complete crossing-order systems extending the core.

Once a planar permutation for $x$ is checked, the process continues by adding a new $C_8$ edge. For permutations found to be nonplanar, all systems $\Pi$ that extend them are eliminated. If the branch resulting from a planar permutation of $x$ dies, then the program backtracks to continue enumerating the remaining permutations for $x$. The computer also helps organize later searches more efficiently through slot classes. A slot is essentially a position in any $\pi_e$ where $x$ could go, and a slot class groups slots that have similar local positions in the orderings relative to the $C_6$ crossings.

This means that, when a new edge is added, all possible crossing-order configurations for the selected edges are checked. If some are planar, then a new edge is added to each of those configurations and tested for planarity. If some of those crossing orders are planar, the same process is repeated for the next edge added to the surviving configurations. This looks like the branching tree shown in Figure~\ref{fig:endpoint-tree}.
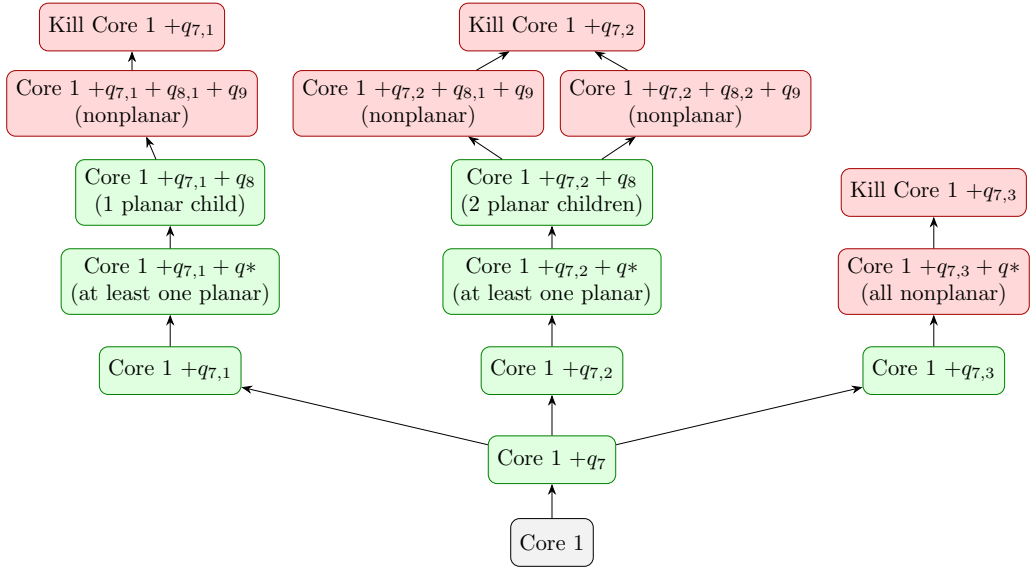
\begin{figure}[H]
\centering

\resizebox{0.82\textwidth}{!}{%
\begin{tikzpicture}[
    >=Stealth,
    every node/.style={
        draw,
        rounded corners,
        align=center,
        minimum height=0.75cm,
        inner sep=4pt,
        font=\small
    },
    safe/.style={fill=green!12, draw=green!50!black},
    dead/.style={fill=red!15, draw=red!65!black},
    root/.style={fill=gray!10, draw=black}
]

% ------------------------------------------------
% Root
% ------------------------------------------------

\node[root] (core1) at (0,0)
{Core 1};

\node[safe] (q7) at (0,1.3)
{Core 1 $+q_7$};

% ------------------------------------------------
% Three q7 signatures
% ------------------------------------------------

\node[safe] (q71) at (-6,2.7)
{Core 1 $+q_{7,1}$};

\node[safe] (q72) at (0,2.7)
{Core 1 $+q_{7,2}$};

\node[safe] (q73) at (6,2.7)
{Core 1 $+q_{7,3}$};

% ------------------------------------------------
% Left branch
% ------------------------------------------------

\node[safe] (q71q6) at (-6,4.1)
{Core 1 $+q_{7,1}+q*$\\
(at least one planar)};

\node[safe] (q71q8) at (-6,5.5)
{Core 1 $+q_{7,1}+q_8$\\
(1 planar child)};

\node[dead] (q71q8q9) at (-6.6,6.9)
{Core 1 $+q_{7,1}+q_{8,1}+q_9$\\
(nonplanar)};

\node[dead] (kill71) at (-6.6,8.1)
{Kill Core 1 $+q_{7,1}$};

% ------------------------------------------------
% Middle branch
% ------------------------------------------------

\node[safe] (q72q6) at (0,4.1)
{Core 1 $+q_{7,2}+q*$\\
(at least one planar)};

\node[safe] (q72q8) at (0,5.5)
{Core 1 $+q_{7,2}+q_8$\\
(2 planar children)};

\node[dead] (q72q81q9) at (-2.1,6.9)
{Core 1 $+q_{7,2}+q_{8,1}+q_9$\\
(nonplanar)};

\node[dead] (q72q82q9) at (2.1,6.9)
{Core 1 $+q_{7,2}+q_{8,2}+q_9$\\
(nonplanar)};

\node[dead] (kill72) at (0,8.1)
{Kill Core 1 $+q_{7,2}$};

% ------------------------------------------------
% Right branch
% ------------------------------------------------

\node[dead] (q73q6) at (6,4.1)
{Core 1 $+q_{7,3}+q*$\\
(all nonplanar)};

\node[dead] (kill73) at (6,5.5)
{Kill Core 1 $+q_{7,3}$};

% ------------------------------------------------
% Arrows
% ------------------------------------------------

\draw[->] (core1) -- (q7);

\draw[->] (q7) -- (q71);
\draw[->] (q7) -- (q72);
\draw[->] (q7) -- (q73);

% Left branch
\draw[->] (q71) -- (q71q6);
\draw[->] (q71q6) -- (q71q8);
\draw[->] (q71q8) -- (q71q8q9);
\draw[->] (q71q8q9) -- (kill71);

% Middle branch
\draw[->] (q72) -- (q72q6);
\draw[->] (q72q6) -- (q72q8);
\draw[->] (q72q8) -- (q72q81q9);
\draw[->] (q72q8) -- (q72q82q9);
\draw[->] (q72q81q9) -- (kill72);
\draw[->] (q72q82q9) -- (kill72);

% Right branch
\draw[->] (q73) -- (q73q6);
\draw[->] (q73q6) -- (kill73);

\end{tikzpicture}%
}

\caption{Schematic illustration of the branching search and endpoint pruning.
Green nodes indicate states for which at least one planar selected
restriction survives. Red nodes indicate states whose continuations
are eliminated. Here, $q$ denotes an edge of $C_8$, and $q^*$ denotes a distinct later edge tested for lookahead.}

\label{fig:endpoint-tree}
\end{figure}
 A major problem here, however, is that the program could end up checking edges up to the last one in a branch, only to find that adding the final edge makes the branch nonplanar. A trick we implemented to reduce this problem in many cases is to test a strategically chosen later edge in a branch as the second edge before continuing with the intended branching order. By doing this, the branch is eliminated only if every compatible completed restriction involving the lookahead edge is nonplanar. If there is even one planar scenario, then the lookahead does not kill the branch, and the normal search continues in the originally intended edge order. This is illustrated in Figure~\ref{fig:endpoint-tree}.
 
 Another strategy we use to speed up this algorithm involves making sure that the crossing-order system we are viewing is compatible with the planar configurations of the $C_6$ and $C_8$. Suppose we are checking a given graph with an ordering $\Pi_i$. If we remove from this ordering the $C_8$ edges that are not part of the $C_6$ and maintain the order of the other edges, we obtain $\Pi_{C_6}$. Likewise, if we instead remove all the $C_6$ edges that are not part of the $C_8$ and maintain the order of the remaining edges, we obtain $\Pi_{C_8}$. In order for $\Pi_i$ to have any hope of eventually becoming a full planar configuration in $G'$, $\Pi_{C_6}$ must already be, or extend to, one of the eight planar $C_6$ cores, and $\Pi_{C_8}$ must already be, or extend to, one of the 2,544 allowed planar $C_8$ crossing-order systems.

As for the number 2,544, the program begins with these $C_8$ configurations, which we previously generated. In the process of generating these 2,544 systems, the $C_8$ edges were labeled cyclically and given the same orientation. We generated 2,544 planar systems using the Fulek--Pach crossing-order and planarity framework and then checked that the generated set was closed under rotations, reflections, and combinatorial adjacent-triple flips. These flips include all genuine Reidemeister III transitions, which can slide an edge past the open crossing of two other edges, thus changing the crossing order. Using these transformations, the 2,544 systems were divided into three connected groups of sizes
 \[
 656+896+992=2{,}544.
 \]
 Misereh and Nikolayevsky show that there are three equivalence classes of thrackle drawings of $C_8$ up to isotopy and Reidemeister III moves (isotopy means bending, stretching, and moving edges without removing crossings)~\cite[Section~4]{MiserehNikolayevsky2018Annular}. Because the checked flip relation includes every genuine Reidemeister III transition, the systems in one group cannot belong to the same equivalence class as the systems in the other groups. Therefore, because these 2,544 systems are planar, they are all of the systems relevant to our search. Outside these possibilities, there are also two nine-edge configurations for which we checked the number of planar possibilities before including them in our program. These two are denoted by
\[
H_2 = C_8 \cup \{e_2\},
\qquad
H_5 = C_8 \cup \{e_5\}.
\]
Each looks like a $C_8$ with an extra edge from the $C_6$ in $DB(6,8,-2)$ attached like a tail. Our preprocessing showed that $H_2$ has $60,196$ allowed planar crossing-order systems. Because $H_5$ is the reflected version of $H_2$, it also has $60,196$ allowed planar crossing-order systems, obtained by reflecting the $H_2$ systems. These provide additional possibilities, beyond just $C_6$ and $C_8$, that are checked in the same way.

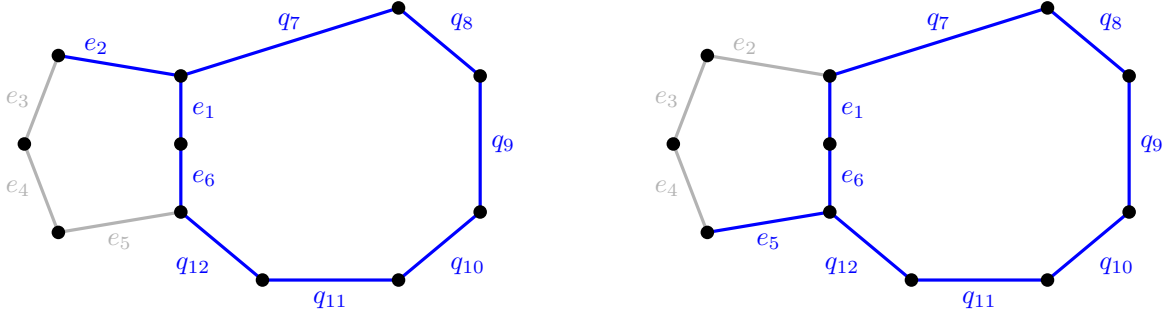
\begin{figure}[H]
\centering

\begin{minipage}{0.48\textwidth}
\centering
\begin{tikzpicture}[
    scale=0.9,
    every node/.style={font=\small},
    vertex/.style={circle, fill=black, inner sep=1.8pt},
    blueedge/.style={blue, very thick},
    dimedge/.style={gray!60, very thick}
]

% Coordinates
\coordinate (A) at (0,1);      % shared top-left
\coordinate (B) at (0,0);      % middle shared
\coordinate (C) at (0,-1);     % shared bottom-left
\coordinate (D) at (-1.8,1.3); % left top
\coordinate (E) at (-2.3,0);   % left middle
\coordinate (F) at (-1.8,-1.3);% left bottom
\coordinate (H) at (3.2,2.0);
\coordinate (I) at (4.4,1.0);
\coordinate (J) at (4.4,-1.0);
\coordinate (K) at (3.2,-2.0);
\coordinate (L) at (1.2,-2.0);

% H2 blue edges
\draw[blueedge] (D)--(A) node[midway, above left] {$e_2$};
\draw[blueedge] (A)--(B) node[midway, right] {$e_1$};
\draw[blueedge] (B)--(C) node[midway, right] {$e_6$};

\draw[blueedge] (A)--(H) node[midway, above] {$q_7$};
\draw[blueedge] (H)--(I) node[midway, above right] {$q_8$};
\draw[blueedge] (I)--(J) node[midway, right] {$q_9$};
\draw[blueedge] (J)--(K) node[midway, below right] {$q_{10}$};
\draw[blueedge] (K)--(L) node[midway, below] {$q_{11}$};
\draw[blueedge] (L)--(C) node[midway, below left] {$q_{12}$};

% H2 dim edges
\draw[dimedge] (D)--(E) node[midway, left] {$e_3$};
\draw[dimedge] (E)--(F) node[midway, left] {$e_4$};
\draw[dimedge] (F)--(C) node[midway, below] {$e_5$};

% Vertices
\foreach \p in {A,B,C,D,E,F,H,I,J,K,L}
    \node[vertex] at (\p) {};

\end{tikzpicture}

\vspace{2mm}
\[
H_2 = C_8 \cup \{e_2\}
\]
\end{minipage}
\hfill
\begin{minipage}{0.48\textwidth}
\centering
\begin{tikzpicture}[
    scale=0.9,
    every node/.style={font=\small},
    vertex/.style={circle, fill=black, inner sep=1.8pt},
    blueedge/.style={blue, very thick},
    dimedge/.style={gray!60, very thick}
]

% Coordinates
\coordinate (A) at (0,1);      
\coordinate (B) at (0,0);      
\coordinate (C) at (0,-1);     
\coordinate (D) at (-1.8,1.3); 
\coordinate (E) at (-2.3,0);   
\coordinate (F) at (-1.8,-1.3);
\coordinate (H) at (3.2,2.0);
\coordinate (I) at (4.4,1.0);
\coordinate (J) at (4.4,-1.0);
\coordinate (K) at (3.2,-2.0);
\coordinate (L) at (1.2,-2.0);

% H5 blue edges
\draw[blueedge] (A)--(B) node[midway, right] {$e_1$};
\draw[blueedge] (B)--(C) node[midway, right] {$e_6$};
\draw[blueedge] (C)--(F) node[midway, below] {$e_5$};

\draw[blueedge] (A)--(H) node[midway, above] {$q_7$};
\draw[blueedge] (H)--(I) node[midway, above right] {$q_8$};
\draw[blueedge] (I)--(J) node[midway, right] {$q_9$};
\draw[blueedge] (J)--(K) node[midway, below right] {$q_{10}$};
\draw[blueedge] (K)--(L) node[midway, below] {$q_{11}$};
\draw[blueedge] (L)--(C) node[midway, below left] {$q_{12}$};

% H5 dim edges
\draw[dimedge] (D)--(A) node[midway, above left] {$e_2$};
\draw[dimedge] (D)--(E) node[midway, left] {$e_3$};
\draw[dimedge] (E)--(F) node[midway, left] {$e_4$};

% Vertices
\foreach \p in {A,B,C,D,E,F,H,I,J,K,L}
    \node[vertex] at (\p) {};

\end{tikzpicture}

\vspace{2mm}
\[
H_5 = C_8 \cup \{e_5\}
\]
\end{minipage}

\caption{The graphs $H_2$ and $H_5$. In each diagram, the blue edges indicate the chosen subgraph, with the $C_8$-only edges labeled $q_7,\dots,q_{12}$ and the $C_6$ edges labeled $e_1,\dots,e_6$. The labels \(e_1,\ldots,e_6\) and \(q_7,\ldots,q_{12}\) in this figure are fixed labels for the \(DB(6,8,-2)\) computation and are separate from the local edge notation used later in Lemma~\ref{lem:four-half-edges}.}
\label{fig:H2H5}
\end{figure}

In order to make the application of this feasible, we use one more strategy. Suppose that $P_8(\Pi_i)$ denotes the set of allowed planar crossing-order systems of $C_8$ to which the current $\Pi_i$ can still extend. Similarly, $P_{H_2}(\Pi_i)$ is the set of remaining $H_2$ systems, and $P_{H_5}(\Pi_i)$ is the set of remaining $H_5$ systems. Instead of checking $\Pi_i$ against every possible $C_8$, $H_2$, and $H_5$ system each time a new edge is added, we check it only against these sets. After $\Pi_i$ is checked, the sets are updated to contain only the orders in the original sets to which $\Pi_i$ could extend. Those new stored sets are then checked against $\Pi_{i+1}$, after which they are updated to store the systems compatible with $\Pi_{i+1}$, and so on. If any set of remaining systems becomes empty, then the current branch cannot extend to a complete planar $\Pi$ and can be eliminated. However, when that branch is eliminated and the program backtracks, the sets of still-possible $C_8$, $H_2$, and $H_5$ systems reset to their previous states at the partial system to which the program backtracked. Additionally, bitsets are used: instead of storing full systems, each set stores ones and zeros that map to pre-stored systems, where ones represent systems to which an extension is still possible and zeros represent systems to which it is not~\cite{ScoppettoThrackleLanguagesv110}.

One final major speedup comes from reflections. A reflection of the $DB(6,8,-2)$ graph pairs the $C_6$ cores
\[
\{0,2\},\quad\{1,4\},\quad\{3,6\},\quad\{5,7\}.
\]
Under this reflection, edge labels are reflected, and crossing-order rows are reversed when the orientation of the corresponding edge is reversed. Here, the orientation is the direction in which edges are traced. This reflection also results in a bijection between the sets of complete extended crossing orders of the paired graphs. Because the graphs are isomorphic, planarity is preserved between the bijected graphs. Additionally, this reflection preserves the applicable $C_8$, $H_2$, and $H_5$ restrictions. This means that searching one core in each pair is sufficient to determine whether the other core admits a planar complete extension; thus, only four out of eight cores need to be checked. Each checked core therefore represents two cores and has a weight of two at the end.

Let $T=(8!)^6(9!)^6$ be the total number of possible complete ordered-crossing systems $\Pi$, and let
\[
U=\frac{T}{46{,}656},
\]
which represents the number of possible complete systems that extend one $C_6$ core. Of the 46,656 $C_6$ cores, 46,648 are excluded as nonplanar by the program. Finally, let $R_i$ be the number of complete systems extending $C_6$ core $i$ that have been resolved by the program, for $i\in\{0,1,3,5\}$. For the program to conclude that the graph is non-thrackleable, it needs
\[
46{,}648U+2(R_0+R_1+R_3+R_5)=T.
\]
This takes the number of complete systems extending the excluded $C_6$ cores and adds the complete systems excluded from the planar $C_6$ cores, where the sum for the four representative cores is doubled to count all systems eliminated across the reflection pairs. Each of these sets is disjoint and contains unique complete systems restricting to a single $C_6$ core. If the program finds a planar complete $\Pi$, it terminates early and returns \textsc{Thrackleable}. If the program ends, is interrupted, or fails without either satisfying this equation or finding a planar complete $\Pi$, it returns \textsc{Inconclusive}. If the program cannot rule out a partial system, it continues searching extensions of that system rather than eliminating it or returning \textsc{Inconclusive}. Figure~\ref{alg:db682-search} summarizes the search structure, but detailed accounting and interruption handling are omitted.

In addition, this is a reconstruction of Fulek and Pach's program with a few adjustments. The program uses parallelism to check the four cores at the same time, and the planarity testing is performed by C++/Boost. The completed run and its accounting were separately audited after the successful run. Expressed as pseudocode, the program looks like the following:
\begin{figure}[!htbp]
\centering
\scriptsize

\begin{minipage}[t]{0.47\textwidth}
\begin{tabbing}
\quad\=\quad\=\quad\=\kill

$\textsc{SearchDB682}():$\\
\> $S \leftarrow E(C_6)$\\
\> $\mathcal{C} \leftarrow \varnothing$\\[1mm]

\> for $i \leftarrow 1$ to $46656$\\
\>\> $\Pi_{S,i} \leftarrow$ $i$th complete\\
\>\> crossing-order system on $S$\\
\>\> if $G'(C_6,\Pi_{S,i})$ is planar\\
\>\>\> $\mathcal{C} \leftarrow
\mathcal{C}\cup\{\Pi_{S,i}\}$\\[1mm]

\> $\mathcal{C} \leftarrow$ one representative\\
\> from each reflection pair\\[1mm]

\> for each core $P\in\mathcal{C}$ in parallel\\
\>\> $Q \leftarrow$ remaining $q$-edges\\
\>\> $\mathrm{result}_P\leftarrow
\textsc{Search}(P,Q,\mathrm{none})$\\
\> if any $\mathrm{result}_P$ contains a verified\\
\>\> planar complete system\\
\>\>\> return \textsc{Thrackleable}\\
\> if every core search finishes negative and\\
\>\> the accounting equation above holds\\
\>\>\> return \textsc{Nonthrackleable}\\

\> return \textsc{Inconclusive}

\end{tabbing}
\end{minipage}
\hspace{0.025\textwidth}
\begin{minipage}[t]{0.49\textwidth}
\begin{tabbing}
\hspace{0.8em}\=\hspace{0.8em}\=\hspace{0.8em}\=
\hspace{0.8em}\=\hspace{0.8em}\=\kill

$\textsc{Search}(P,Q,q_{\mathrm{current}}):$\\
\> if $q_{\mathrm{current}}$ has lookahead edge $q^*$\\
\>\> check every compatible completed\\
\>\> restriction involving $q^*$\\
\>\> if every such restriction is nonplanar\\
\>\>\> return \textsc{Dead}\\[1mm]

\> if $Q=\varnothing$\\
\>\> if $G'(DB(6,8,-2),P)$ is planar\\
\>\>\> independently verify the complete witness\\
\>\>\> if verification succeeds\\
\>\>\>\> return \textsc{Thrackleable}\\
\>\> return \textsc{Dead}\\[1mm]

\> $q\leftarrow$ next edge in $Q$\\

\> for each exact extension $\pi_q$ involving $q$\\
\>\> $P'\leftarrow P$ with this exact extension applied\\
\>\>\> to $q$ and all affected existing rows\\
\>\> during this construction, form complete selected\\
\>\>\> restrictions by omitting still-pending additional edges\\
\>\> check each resulting augmented graph for planarity\\
\>\> check extendability to $C_8,H_2,H_5$\\

\>\> if $P'$ passes all checks\\
\>\>\> $\mathrm{result}\leftarrow
\textsc{Search}(P',Q\setminus\{q\},q)$\\
\>\>\> if $\mathrm{result}=\textsc{Thrackleable}$\\
\>\>\>\> return \textsc{Thrackleable}\\
\>\>\> if $\mathrm{result}=\textsc{Inconclusive}$\\
\>\>\>\> return \textsc{Inconclusive}\\[1mm]

\> return \textsc{Dead}

\end{tabbing}
\end{minipage}

\caption{High-level recursive pseudocode for the adapted Fulek--Pach
search on $DB(6,8,-2)$.}
\label{alg:db682-search}
\end{figure}
\FloatBarrier
\begin{lemma}[Computer-assisted]
The dumbbell $DB(6,8,-2)$ is not thrackleable.
\end{lemma}
\begin{proof}
In order for an abstract graph $G$ to be thrackleable, there must exist a complete $\Pi=\{\pi_e\mid e\in E(G)\}$ such that $G'$ is planar~\cite{FulekPach2011}. In order to certify non-thrackleability, $100.0\%$ of the possible configurations for $\Pi$ must be safely eliminated or completely tested, and none may give a planar $G'$. Our algorithm establishes the following:
\[
\text{all possible }\Pi
\longrightarrow
\begin{cases}
\text{safely pruned},\\
\text{or fully tested}.
\end{cases}
\]

After running our adapted algorithm, we obtain the following results:
 \begin{center}
\begin{tabular}{lr}
\hline
Result & \texttt{NONTHRACKLEABLE} \\
Resolved search space & $100\%$ \\
Total $C_6$ crossing-order systems & $46,\!656$ \\
Nonplanar $C_6$ systems excluded & $46,\!648$ \\
Planar $C_6$ cores retained & $8$ \\
Reflection representatives searched & $4$ \\
Recursive calls & $853,\!136,\!768$ \\
Complete selected-restriction graphs tested & $7,\!547,\!373$ \\
Main-search locked-prefix planarity tests & $386$ \\
Partial nonplanarity prunes & $0$ \\
Full $\Pi$ planarity tests & $0$ \\
Runtime & $9467.65$ seconds \\
\hline
\end{tabular}
\end{center}

The frozen program used to obtain these results was
\texttt{fulek\_pach\_thrackle\_v0150\_FINAL\_23ae90.py}, with the SHA-256 hash
\texttt{23ae908626820dd3afaf1dbbc180619186e8bfcb417bf8a2bbcac2bb76369632}. It was run with the command
\begin{Verbatim}[breaklines=true,breakanywhere=true]
python fulek_pach_thrackle_v0150_FINAL_23ae90.py search-db682-cores --processes 4 --planarity-backend boost --report db682_run.json
\end{Verbatim}
This table reports the original successful run, while a later successful run took $12,063.37$ seconds under the following conditions:
\begin{itemize}
    \item Windows 11
    \item Python 3.12.7
    \item NetworkX 3.3
    \item Boost 1.82
    \item MSVC 19.44 x64
    \item Four search workers
\end{itemize}
Because $100.0\%$ of the possible systems $\Pi$ were either safely checked or safely pruned, and no $\Pi$ resulted in a planar $G'$, the Fulek--Pach equivalence implies that \(DB(6,8,-2)\) is not thrackleable~\cite{ScoppettoDB682v0150}.
\end{proof}
\subsection{Comparison of V0.15.0 to Fulek--Pach Reconstruction}
In a controlled benchmark, both our independent reconstruction of the published Fulek–Pach method and v0.15.0 were initialized with the same $C_6$ core, and progress was checked over the 0--2, 2--5, 5--10, and 10--15 minute intervals for each core. Each of the eight runs was done sequentially with one worker and none finishing. Following the initial bulk exclusion by each program in the first two minutes, the observed search-space resolution-rate ratios across the four tested cores were approximately $1.74\times10^{24}$, $1.63\times10^{20}$, and $1.42\times10^{20}$ over the 2--5, 5--10, and 10--15 minute intervals, respectively. After 15 minutes, v0.15.0 had resolved an average of $99.9782\%$ of the extension space, compared with $23.0981\%$ for the reconstruction. These measurements demonstrate a substantially greater observed search-space resolution rate during the measured post-warmup intervals~\cite{ScoppettoSeededBenchmark2026}.
\section{Consequences of the Dumbbell Exclusion}
Using this exclusion of the $DB(6,8,-2)$ we are able to improve Xu's local bound around 8-faces. 
\begin{definition}
    A half-edge block is a maximal segment of consecutive half-edges belonging to the same 6-face.
\end{definition}
\begin{definition}
    An edge bridge occurs when a cycle shares a certain number of half-edges with another cycle and, after the two cycles separate, shares half-edges with it again in a separate block.
\end{definition}
\begin{definition}
    A vertex bridge occurs when a cycle shares a certain number of half-edges with another cycle and, after the two cycles separate, shares another vertex with it that is separate from the vertices already shared.
\end{definition}
For the following face arguments, let $G$ be a 2-connected, triangle-free, thrackleable graph with a cellular parity embedding in the projective plane when $G$ is nonbipartite and a plane embedding when it is bipartite. More specifically, a parity embedding is a special embedding in the projective plane in which all even cycles are two-sided: a thin neighborhood extending from a cycle behaves like an annulus (see Figure~\ref{fig:annulus-neighborhood}), so traveling once around one local side of the cycle returns to the same local side. All odd cycles are one-sided: instead of behaving like an annulus, a thin neighborhood extending from such a cycle behaves like a M\"obius strip, so traveling once around one local side of the cycle leads to what appears to be the opposite local side without crossing the cycle~\cite[Section~2, Corollary~1]{FulekPach2019}.
\begin{figure}[H]
\centering
\begin{tikzpicture}[>=Stealth, line cap=round, line join=round]

%------------------------------------------------
% Dimensions
%------------------------------------------------
\def\W{4.4}      % width of cylinder
\def\H{3.6}      % height of cylinder
\def\R{0.55}     % vertical radius of top/bottom ellipses

%------------------------------------------------
% Cylinder / annulus band
%------------------------------------------------
\fill[blue!20]
(-2.2,1.8) arc[start angle=180,end angle=360,x radius=2.2,y radius=\R]
-- (2.2,-1.8)
arc[start angle=0,end angle=180,x radius=2.2,y radius=\R]
-- cycle;

% Side boundaries
\draw[blue!55!black, thick] (-2.2,1.8) -- (-2.2,-1.8);
\draw[blue!55!black, thick] ( 2.2,1.8) -- ( 2.2,-1.8);

% Top ellipse
\draw[blue!55!black, thick, dashed] (-2.2,1.8) arc[start angle=180,end angle=360,x radius=2.2,y radius=\R];
\draw[blue!55!black, thick]        ( 2.2,1.8) arc[start angle=0,end angle=180,x radius=2.2,y radius=\R];

% Bottom ellipse
\draw[blue!55!black, thick]        (-2.2,-1.8) arc[start angle=180,end angle=360,x radius=2.2,y radius=\R];
\draw[blue!55!black, thick, dashed]( 2.2,-1.8) arc[start angle=0,end angle=180,x radius=2.2,y radius=\R];

% Middle cycle
\draw[blue!75!black, very thick]
(-2.2,0) arc[start angle=180,end angle=360,x radius=2.2,y radius=\R];
\draw[blue!75!black, very thick, dashed]
(2.2,0) arc[start angle=0,end angle=180,x radius=2.2,y radius=\R];

%------------------------------------------------
% Left labels (spread out more)
%------------------------------------------------
\node[align=right] (toplabel) at (-5.2,  1.95) {\small One side\\[-1mm]\small of the band};
\node[align=right] (midlabel) at (-5.2,  0.10) {\small The cycle\\[-1mm]\small (entire loop)};
\node[align=right] (botlabel) at (-5.2, -1.75) {\small The other side\\[-1mm]\small of the band};

\draw[->, thin] (toplabel.east) -- (-2.3, 1.45);
\draw[->, thin] (midlabel.east) -- (-2.3, 0.05);
\draw[->, thin] (botlabel.east) -- (-2.3,-1.45);

%------------------------------------------------
% Right label
%------------------------------------------------
\node[draw, fill=white, align=center] (rightlabel) at (4.6,0.15)
{\small Two distinct\\[-1mm]\small sides globally};
\draw[->, thin] (rightlabel.west) -- (2.35,0.15);

\end{tikzpicture}
\caption{Annulus neighborhood of a two-sided cycle.}
\label{fig:annulus-neighborhood}
\end{figure}
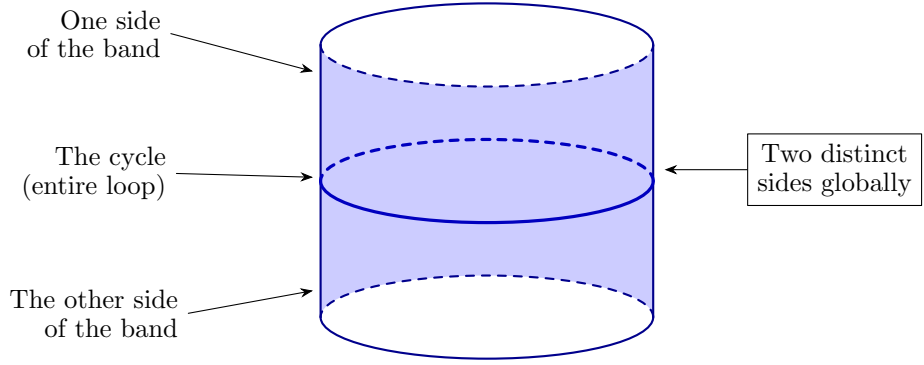
\begin{lemma}[Fulek--Pach~\cite{FulekPach2019}]
    In a thrackleable graph, two distinct 6-cycles cannot share a vertex.
\end{lemma}
\begin{lemma}[Fulek--Pach, Proposition~1~\cite{FulekPach2019}]
    For a 2-connected triangle-free thrackleable graph with a plane embedding in the bipartite case and a cellular parity embedding in the projective plane in the nonbipartite case, the boundary walk of an $8$-face is a $C_8$, the boundary walk of a 6-face is a $C_6$, and every face has degree at least six.
\end{lemma}
\begin{lemma}\label{lem:four-half-edges}
    Every 8-face has at most four half-edges.
\end{lemma}
\begin{proof}
    In an 8-face with five half-edges, there are three non-half-edges, since
    \[
8 - 5 = 3.
\]
Because of this, the half-edges can be split into at most three half-edge blocks. Since those five half-edges are distributed among at most three half-edge blocks, the pigeonhole principle gives
\[
\left\lceil \frac{5}{3} \right\rceil = 2.
\]
Thus, at least one of the three half-edge blocks contains at least two half-edges in a row, giving $s\ge2$, where $s$ is the number of half-edges in a half-edge block. If each 6-face meets the 8-face only at the edges and vertices of one half-edge block, then, because each half-edge block can belong to only one 6-face, there is at least one $DB(6,8,\ell)$ in an 8-face with five half-edges, where
\[
2 \leq s \leq 5,\qquad \ell=-s,
\]
hence
\[
-5 \leq \ell \leq -2.
\]
Because all $DB(6,8,\ell)$ for these values of $\ell$ are non-thrackleable, an 8-face cannot have five half-edges if each 6-face meets the 8-face at only one half-edge block. It remains, however, to consider vertex bridges and edge bridges. The possible half-edge-block combinations are
 \[
5, 4+1, 3+2, 3+1+1, 2+2+1.
\]
Starting with edge bridges, we can first eliminate the pattern $5$, since an edge bridge requires two half-edge blocks, but this pattern has only one. We are left with the following eight possible edge bridges:
\[
\begin{array}{c|c}
\text{Global block pattern} & \text{Possible blocks owned by one } C_6 \\ \hline
4+1 & 4+1 \\
3+2 & 3+2 \\
3+1+1 & 3+1,\ 1+1,\ 3+1+1 \\
2+2+1 & 2+2,\ 2+1,\ 2+2+1\\
\end{array}
\]
Because a 6-face requires at least two edges to make an edge bridge and also uses some edges as half-edges, the number of edges in the 6-cycle left to construct the bridge after removing the half-edges must be at least two. We use the formula
\[
6-j \ge k,
\]
to represent this, where $j$ is the number of half-edges on the edge bridge and $k$ is the minimum number of edges in the 6-face required to connect the blocks in an edge bridge. Using this, we can eliminate two groups of scenarios:
\[
4+1,\ 3+2:\qquad j=5,\quad k=2,\quad 1<2,
\]

and

\[
3+1+1,\ 2+2+1:\qquad j=5,\quad k=3,\quad 1<3,
\]
leaving the following possible configurations:
\[
3+1, 1+1, 2+2, 2+1.
\]
We immediately eliminate $1+1$, since it occurs in the $3+1+1$ pattern. This means there is an extra block of three half-edges with its own 6-face that cannot form an edge bridge to another half-edge block, since the other blocks are already part of the $1+1$ bridge. Additionally, the 6-face belonging to the block of three half-edges cannot form a vertex bridge. In a $3+1+1$ pattern, the five half-edges and three non-half-edge separators ensure that every vertex already has a 6-face attached. Because two distinct $C_6$ cycles cannot share a vertex, there must therefore be a non-thrackleable $DB(6,8,-3)$.

We can also eliminate $2+1$, since it occurs in the $2+2+1$ pattern. This means there is an extra block of two half-edges with its own 6-face that cannot form an edge bridge to another half-edge block, since the other blocks are already part of the $2+1$ bridge. For the same reason as in the $3+1+1$ case, the remaining block of two half-edges cannot form a vertex bridge. Because the vertices of two distinct $C_6$ cycles cannot meet, this would require a non-thrackleable $DB(6,8,-2)$. Now, all we need to check are
\[
3+1 \text{ and } 2+2.
\]
Starting with $2+2$, we assign the vertices of the 8-face as
\[
F=v_1v_2v_3v_4v_5v_6v_7v_8,
\]
where
\[
e_i=v_iv_{i+1}\qquad (1\leq i\leq 7),
\qquad
e_8=v_8v_1.
\]
According to the $2+2+1$ configuration, $\{e_1,e_2\}$ are the first two half-edges, $\{e_4,e_5\}$ are the next two half-edges, and $\{e_7\}$ is the last half-edge. To make the $2+2$ edge bridge, we use $\{e_1,e_2\}$ and $\{e_4,e_5\}$ as the half-edges. Connecting them requires at least two edges, yielding the following two ways to connect the two half-edge blocks:
\[
\{v_1v_4,v_3v_6\} \text{ and }  \{v_1v_6,v_3v_4\}.
\]
We can first eliminate the second possibility because $v_3v_4=e_3$ is an edge of the 8-face. This would make another half-edge, creating a sixth half-edge and contradicting the assumption $h(F)=5$. Next, we can eliminate the first possibility because it creates $v_1v_2v_3v_4v_1$, which is a non-thrackleable $C_4$.

Next, we eliminate the edge-bridge possibilities for $3+1$ using the same vertex assignments as in the $2+2$ proof. According to the $3+1+1$ configuration, $\{e_1,e_2,e_3\}$ are the first three half-edges, $\{e_5\}$ is the next half-edge, and $\{e_7\}$ is the last half-edge. To make the $3+1$ edge bridge, we use $\{e_1,e_2,e_3\}$ and $\{e_5\}$ as the half-edges. Connecting them requires at least two edges, yielding the following two ways to connect the two half-edge blocks:
\[
\{v_1v_5,v_4v_6\} \text{ and }  \{v_1v_6,v_4v_5\}.
\]
We can first eliminate the second possibility because it contains $v_4v_5=e_4$, an edge of the 8-face. This would make another half-edge, creating a sixth half-edge and contradicting the assumption $h(F)=5$. Next, we can eliminate the first possibility because it creates the $C_6$ $v_1v_5v_4v_6v_7v_8v_1$, which shares a vertex with the $C_6$ $v_1v_2v_3v_4v_6v_5v_1$. This is non-thrackleable because two distinct 6-cycles cannot share a vertex.

Finally, we consider vertex bridges. The only configurations that could have a vertex bridge are those with at least two non-half-edges in a row. This is because a vertex bridge cannot meet a vertex that is incident to a half-edge, as this would result in two 6-faces sharing a vertex. Thus, the vertex bridge must connect to an isolated vertex, meaning a vertex on $C_8$ incident to no half-edges. Since the $2+2+1$ and $3+1+1$ patterns do not have at least two consecutive non-half-edges, this leaves only
\[
3+2, 4+1, 5
\]
and within these patterns, the possible blocks forming the base of a vertex bridge are
\[
\begin{array}{c|c}
\text{Global block pattern} & \text{Possible vertex-bridge block} \\ \hline
5 & 5 \\
4+1 & 4 \text{ or } 1 \\
3+2 & 3 \text{ or } 2
\end{array}
\]
We exclude the pattern $5$ immediately, because five half-edges already contain six vertices, and no seventh vertex can belong to the same 6-face. Additionally, an internal vertex of a half-edge block is already incident to edges of the $C_6$ that owns the block, so the remaining edges of that $C_6$ must leave from the endpoints of the block.

Starting with $4+1$ and using the same 8-face vertex setup as before, the first four half-edges are $\{e_1,e_2,e_3,e_4\}$, and the next half-edge is $\{e_6\}$. We begin by making the block of four half-edges part of our vertex-bridging 6-face, which means that the bridge already contains $\{v_1,v_2,v_3,v_4,v_5\}$. It must now use its two remaining edges to connect the endpoints of this path to a vertex $x$ on the 8-face, giving
\[
\{v_5x,xv_1\}.
\]
Because distinct $C_6$ cycles cannot share vertices, and the boundary vertices of the other half-edge block are $v_6$ and $v_7$, the vertex to which we are bridging cannot be either of these. Otherwise, two 6-cycles would share a vertex. This leaves only $v_8$, so $x=v_8$, resulting in
\[
\{v_5v_8,v_8v_1\}.
\]
However, $v_8v_1=e_8$ is an edge of the 8-face, thereby adding an extra half-edge and violating our assumption that the graph has only five half-edges. If, on the other hand, we use the one-half-edge block in our vertex bridge instead of the four-half-edge block, it cannot bridge to a vertex on the four-half-edge block because two distinct $C_6$ cycles would then share a vertex. There is also only one isolated vertex, so the 6-face on the four-half-edge block cannot form a bridge if the one-half-edge block forms a vertex bridge. The graph is therefore guaranteed to contain the non-thrackleable $DB(6,8,-4)$ on the remaining four-half-edge block.

Next, consider the $3+2$ configuration. If the three-half-edge block is included in the vertex bridge, then the 6-face on the two-half-edge block cannot form a bridge because there is only one isolated vertex and two distinct $C_6$ cycles cannot share a vertex. There is therefore automatically a non-thrackleable $DB(6,8,-2)$ on the remaining two-half-edge block. If the two-half-edge block is included in the vertex bridge, then, for the same reason, the 6-face on the three-half-edge block cannot form a bridge. There is therefore automatically a non-thrackleable $DB(6,8,-3)$ on the remaining three-half-edge block. The preceding work gives
\[
h(F)\neq 5.
\]
Xu's Lemma~2.1~\cite{Xu2021} gives
\[
h(F)\leq 5.
\]
Therefore,
\[
h(F)\leq 5 \qquad \text{and} \qquad h(F)\neq 5
\quad\Longrightarrow\quad
h(F)\leq 4.
\]

\end{proof}
\section{Main Result}
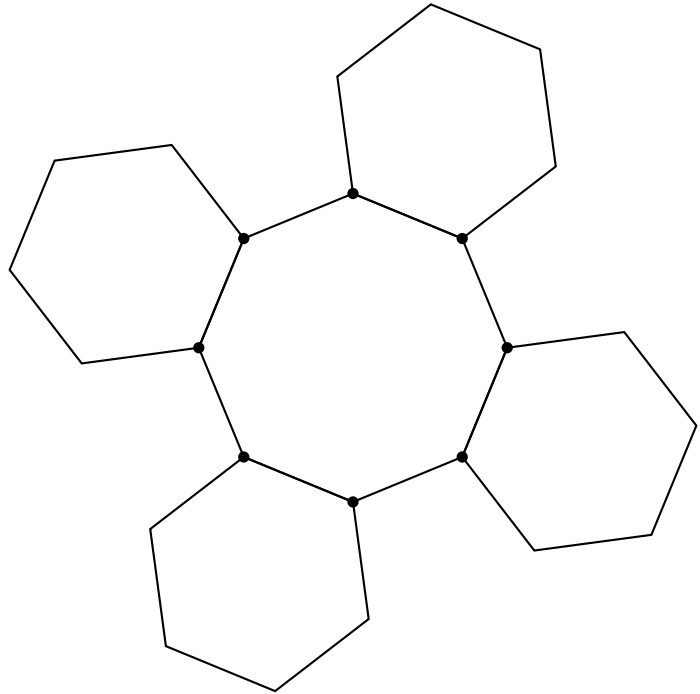
\begin{figure}[H]
\centering
\begin{tikzpicture}[scale=0.85,
    every node/.style={font=\small},
    vertex/.style={circle, fill=black, inner sep=1.5pt}
]

% Central 8-face vertices
\coordinate (v1) at (0,2.4);
\coordinate (v2) at (1.7,1.7);
\coordinate (v3) at (2.4,0);
\coordinate (v4) at (1.7,-1.7);
\coordinate (v5) at (0,-2.4);
\coordinate (v6) at (-1.7,-1.7);
\coordinate (v7) at (-2.4,0);
\coordinate (v8) at (-1.7,1.7);

% Central 8-cycle
\draw[thick]
(v1)--(v2)--(v3)--(v4)--(v5)--(v6)--(v7)--(v8)--cycle;

% ---------- NE 6-face sharing v1v2 ----------
\coordinate (a1) at (3.156,2.822);
\coordinate (a2) at (2.912,4.644);
\coordinate (a3) at (1.212,5.344);
\coordinate (a4) at (-0.244,4.222);

\draw[thick]
(v1)--(v2)--(a1)--(a2)--(a3)--(a4)--cycle;

% ---------- SE 6-face sharing v3v4 ----------
\coordinate (b1) at (2.822,-3.156);
\coordinate (b2) at (4.644,-2.912);
\coordinate (b3) at (5.344,-1.212);
\coordinate (b4) at (4.222,0.244);

\draw[thick]
(v3)--(v4)--(b1)--(b2)--(b3)--(b4)--cycle;

% ---------- SW 6-face sharing v5v6 ----------
\coordinate (c1) at (-3.156,-2.822);
\coordinate (c2) at (-2.912,-4.644);
\coordinate (c3) at (-1.212,-5.344);
\coordinate (c4) at (0.244,-4.222);

\draw[thick]
(v5)--(v6)--(c1)--(c2)--(c3)--(c4)--cycle;

% ---------- NW 6-face sharing v7v8 ----------
\coordinate (d1) at (-2.822,3.156);
\coordinate (d2) at (-4.644,2.912);
\coordinate (d3) at (-5.344,1.212);
\coordinate (d4) at (-4.222,-0.244);

\draw[thick]
(v7)--(v8)--(d1)--(d2)--(d3)--(d4)--cycle;

% Vertex dots on central 8-face
\foreach \i in {1,...,8}
    \node[vertex] at (v\i) {};

% H/N labels
\node at (0.9,2.25) {};
\node at (2.2,0.9) {};
\node at (2.2,-0.9) {};
\node at (0.9,-2.25) {};
\node at (-0.9,-2.25) {};
\node at (-2.2,-0.9) {};
\node at (-2.2,0.9) {};
\node at (-0.9,2.25) {};

% Central face label
\node at (0,0) {};

\end{tikzpicture}
\caption{An $8$-face with alternating half-edges and non-half-edges,
giving the pattern $HNHNHNHN$.}
\end{figure}
\begin{proof}[\textbf{\upshape Proof of Theorem~\ref{thm:main-bound}}]
The formula $m\le\frac{18}{13}(n-1)$ fails in one case: the three-vertex triangle $C_3$. This triangle would require
\[
m\le\frac{36}{13},
\]
which is false. Importantly, for \(1\le n\le3\), since \(G\) is simple,
\[
m\le \binom{n}{2}\le n.
\]
The second bound cannot be extended to \(n=3\), since for \(C_3\),
\[
3>\frac{18}{13}(3-1)=\frac{36}{13}.
\]
Thus, we may assume \(n\ge4\).

\noindent\textit{Case 1: Connected Triangle-Free Case.}\\

For this proof, we adapt Xu's 2021 framework~\cite{Xu2021}. Let $G$ be a connected, triangle-free counterexample satisfying $m>\frac{18}{13}(n-1)$ and having the minimum possible number of vertices. Note that triangle-free graphs with $n\le3$ already satisfy our bound.

First, assume that a leaf exists in this graph, where a leaf is a vertex with only one edge attached to it. If we remove this leaf, we have a smaller triangle-free graph that satisfies our bound. Adding the leaf back adds one edge and one vertex, but our bound allows for an increase of $\frac{18}{13}>1$. Thus, a leaf cannot break the bound, and a minimal counterexample cannot have a leaf.

Next, assume that our graph contains a cut vertex $v$, where a cut vertex is a vertex whose removal disconnects the graph into more than one component. If we split our graph into these components with $v$ attached to each, every piece is smaller than our original graph and thus satisfies $m_i\le\frac{18}{13}(n_i-1)$. Summing these inequalities gives
\[
\sum_i m_i
\le \frac{18}{13}\sum_i(n_i-1).
\]
Using $\sum_i(n_i-1)=n-1$, we obtain
\[
m\le\frac{18}{13}(n-1),
\]
meaning that $G$ satisfies the bound, which contradicts the assumption that it is a counterexample. Because $G$ is connected and has no cut vertices, $G$ is 2-connected.

For $v\in V$, let $d(v)$ be its degree, where $V$ is the set of vertices in $G$, and let $\mathcal{F}$ be the set of faces in $G$. Define $t(v)=d(v)-r$ for all $v\in V$ and $t(u)=\frac{r-2}{2}d(u)-r$ for all $u\in\mathcal{F}$. Finally, for the following discharging calculation, let $c$ be a temporary coefficient and let $r=2c$.

\textit{Charge Setup.}
Since $G$ is 2-connected,
\[
\begin{aligned}
r+\sum_{v\in V} t(v)+\sum_{U\in \mathcal{F}} t(U)
&=
\left(2m-(n-1)r\right)
+\frac{r-2}{2}\cdot 2m-fr \\
&=
mr-nr-fr+r \\
&=
-(k-1)r,
\end{aligned}
\]

where
\[
f=|\mathcal{F}|,
\]
and
\[
k=
\begin{cases}
2, & \text{if } G \text{ is bipartite},\\
1, & \text{if } G \text{ is nonbipartite}.
\end{cases}
\]

\textit{Discharging Rule: Every $8^+$-face distributes an equal charge $x$ through each of its half-edges.}

\textit{Choosing the discharging parameters.}

A 6-face begins with a charge of
\[
t(6)=\frac{r-2}{2}(6)-r=2r-6.
\]
It receives charge $x$ through each of its six edges, and because the final charge must be nonnegative, we have
\[
t^*(6)=2r-6+6x \ge 0.
\]
This is because the whole graph cannot be a $C_6$, since
\[
6\le\frac{18}{13}(6-1)=\frac{90}{13},
\]
which is true, meaning that $C_6$ satisfies the bound and cannot be a counterexample. Additionally, suppose we pick an edge on a 6-face. There must be another face on the other side of that edge, and if it is a 6-face, there are two possibilities. The first is that the two faces share the same boundary, forming an inside 6-face and an outside 6-face. In this scenario, the whole graph is a $C_6$, which we already proved impossible. The other possibility is that the two 6-faces are distinct and share certain vertices, but this is also impossible because two distinct $C_6$ cycles cannot share a vertex. Thus, because every face has even length at least six and an adjacent 6-face is impossible, every 6-face must be adjacent only to $8^+$-faces. Therefore,
\[
x\ge1-\frac{r}{3}.
\]
The 8-face begins with a charge of
\[
t(8)=\frac{r-2}{2}(8)-r=3r-8.
\]
By Lemma~\ref{lem:four-half-edges}, an 8-face has at most four half-edges, so it distributes an equal amount of charge through at most four half-edges, giving
\[
t^*(8)\ge3r-8-4x \ge 0,
\]
which becomes
\[
\frac{3r-8}{4} \ge x.
\]
Substituting, we get
\[
\frac{3r-8}{4}\ge1-\frac{r}{3},
\]
simplifying to
\[
r\ge\frac{36}{13}.
\]
Finally, because $r=2c$, we substitute to get
\[
2c\ge\frac{36}{13},
\]
which simplifies to
\[
c\ge\frac{18}{13}.
\]
To obtain the strongest bound from this discharging rule, we take $c=\frac{18}{13}$. From this, we also know $x$ to be
\[
x=1-\frac{r}{3}=\frac{1}{13}.
\]
\textit{Discharging.}
Looking at the charges again, we see that $t(8)=3(2(\frac{18}{13}))-8=\frac{4}{13}$ and $t(6)=2(2(\frac{18}{13}))-6=-\frac{6}{13}$. Using the discharging rule, we see that for 6-faces,
\[
t^*(6)=t(6)+6(\frac{1}{13})=0.
\]
For 8-faces,
\[
t^*(8)\ge t(8)-4(\frac{1}{13})=0.
\]
For a $10^+$-face,
\[
\begin{aligned}
t^*(u)
&\ge t(u)-d(u)\left(\frac{1}{13}\right)\\
&=\frac{5}{13}d(u)-d(u)\left(\frac{1}{13}\right)-\frac{36}{13}.
\end{aligned}
\]
Since $d(u)\ge10$,
\[
t^*(u)\ge\frac{4}{13}>0.
\]
Thus, $\sum_{u\in\mathcal{F}}t^*(u)\ge0$, and $t(v)=t^*(v)$ for each $v\in V$. We also have
\[
r+\sum_{v\in V}t(v)=2m-(n-1)r>0.
\]
This implies that
\[
r+\sum_{v\in V} t^*(v)+\sum_{u\in\mathcal{F}} t^*(u)>0.
\]

\textit{Final contradiction.}

Notice that
\[
\begin{aligned}
0
&\ge -(k-1)r \\
&= r+\sum_{v\in V} t(v)+\sum_{u\in\mathcal{F}} t(u) \\
&= r+\sum_{v\in V} t^*(v)+\sum_{u\in\mathcal{F}} t^*(u) \\
&>0,
\end{aligned}
\]

which is impossible. This completes the triangle-free case.

\noindent\textit{Case 2: Connected Triangle-Containing Case.}\\

We start with a connected, bipartite, thrackleable graph $H$ that contains a cycle and is not equal to $C_6$. Build $H$ block by block, beginning with a cyclic block and attaching each new block $B$ at a cut vertex of $H$. A block is either a maximal 2-connected subgraph or a bridge $K_2$, which is just a single edge. In this scenario, let $N=|V(H)|$, $M=|E(H)|$, $v=|V(B)|$, and $e=|E(B)|$. Because $B$ is connected to the existing graph at one vertex, the increase in the number of vertices is
\[
v-1
\]
while the increase in the number of edges is the full amount $e$. Now define
\[
D(H)=\frac{18}{13}(N-2)-M
\]
to measure how far below the bound $H$ is, where a negative value of $D(H)$ means that $H$ is above the bound and a positive value means that it is below. This is derived from
\[
M\le\frac{18}{13}(N-2).
\]
The charge calculation from Case 1 can be applied to any 2-connected, bipartite, thrackleable graph other than $C_6$. Since the graph is bipartite, it uses $k=2$, so the total charge is $-r$. Since the total face charge is nonnegative, the charge identity gives
\[
2e-r(v-1)\le-r
\quad\Longrightarrow\quad
2e\le r(v-2),
\]
and because $r=\frac{36}{13}$,
\[
e\le\frac{18}{13}(v-2).
\]
This inequality applies because a cyclic block $B\neq C_6$ is 2-connected and, as a subgraph of $H$, is also bipartite and thrackleable. The change in $D(H)$ when a new block is attached is
\[
\Delta D=\frac{18}{13}(v-1)-e.
\]
Because $e\le\frac{18}{13}(v-2)$,
\[
\Delta D\ge\frac{18}{13}(v-1)-\frac{18}{13}(v-2)=\frac{18}{13}.
\]
Thus, $D(H)$ always increases by at least $\frac{18}{13}$. We must now check the two special cases, $K_2$ and $C_6$, which are both bipartite but fail this inequality. For these cases,
\[
K_2:\qquad
\Delta D
=
\frac{18}{13}(1)-1
=
\frac{5}{13},
\]

\[
C_6:\qquad
\Delta D
=
\frac{18}{13}(5)-6
=
\frac{12}{13}.
\]
Thus, they increase $D(H)$ by at least $\frac{5}{13}$. For the only problematic starting block, $C_6$,
\[
D(C_6)
=
\frac{18}{13}(6-2)-6
=
-\frac{6}{13}.
\]
Since \(H\) contains a cycle, choose a cyclic block as the initial block. Because \(H\neq C_6\), at least one additional block must be attached, and the change is known to be at least $\frac{5}{13}$. Therefore,
\[
-\frac{6}{13}+\frac{5}{13}
=
-\frac{1}{13},
\]
meaning that
\[
D(H)\ge-\frac{1}{13}.
\]
Substituting back, we get
\[
\frac{18}{13}(N-2)-M\ge-\frac{1}{13}
\quad\Longrightarrow\quad
M\le\frac{18}{13}(N-2)+\frac{1}{13}.
\]
Now, assume that a connected graph $G$ with $n\ge4$ contains a triangle. Conway's doubling replaces the $C_3$ with a $C_6$ while preserving how the rest of the graph is attached to it. This produces a graph $H$ that is connected, bipartite, and remains thrackleable by the Conway doubling lemma~\cite[Lemma~5]{FulekPach2011}. In $H$, the $C_3$ becomes a $C_6$, so
\[
N=n+3
\qquad\text{and}\qquad
M=m+3,
\]
where $N$ and $M$ are the numbers of vertices and edges in $H$, while $n$ and $m$ are the numbers of vertices and edges in $G$. Since $n\ge4$ in $G$, we have $N\ge4+3=7$, so $H\neq C_6$ and the inequality
\[
M\le\frac{18}{13}(N-2)+\frac{1}{13}
\]
applies. Substituting, we get
\[
\begin{aligned}
m+3
&\le\frac{18}{13}\bigl((n+3)-2\bigr)+\frac{1}{13}\\
&\Longrightarrow m\le\frac{18n+18+1-39}{13}\\
&\Longrightarrow m\le\frac{18n-20}{13}.
\end{aligned}
\]
Because
\[
\frac{18}{13}(n-1)=\frac{18n-18}{13},
\]
we have
\[
\frac{18n-20}{13}<\frac{18n-18}{13},
\]
giving us the desired bound $m\le\frac{18}{13}(n-1)$.

\noindent\textit{Case 3: Disconnected Case.}\\

Finally, consider the case in which $G$ is disconnected. Assume that the graph has at least two connected components, so $j\ge2$, and denote them by $G_1,\ldots,G_j$, where $m_i$ is the number of edges in $G_i$ and $n_i$ is the number of vertices. Every component other than $C_3$ satisfies
\[
m_i\le\frac{18}{13}(n_i-1).
\]
If a $C_3$ component is present, though, then using the same framework as earlier for the $C_6$ and $K_2$ cases, we see that it exceeds the componentwise bound by $\frac{3}{13}$. This means that, for our graph,
\[
\sum_{i=1}^{j}m_i
\le
\sum_{i=1}^{j}
\frac{18}{13}(n_i-1)
+
\frac{3t}{13}
\quad\Longrightarrow\quad
m\le\frac{18}{13}(n-j)+\frac{3t}{13},
\]
where $t$ is the number of components $G_i$ that are exactly $C_3$. We want the form $\frac{18}{13}(n-1)$ rather than $\frac{18}{13}(n-j)$. The difference between these two terms is
\[
\frac{18}{13}(n-1)-\frac{18}{13}(n-j)=\frac{18}{13}(j-1),
\]
meaning that, in the case including triangle components, we have $\frac{18}{13}(j-1)$ of excess available. We know that $t\le j$, because there cannot be more triangle components in our graph than the total number of components. Thus,
\[
3t\le3j\le18(j-1)
\quad\Longrightarrow\quad
\frac{3t}{13}\le\frac{18}{13}(j-1).
\]
This means that the triangles in the graph are not enough to consume all of the excess, giving
\[
m
\le
\frac{18}{13}(n-j)
+
\frac{18}{13}(j-1)
=
\frac{18}{13}(n-1).
\]
\end{proof}

\section{Future Work}

\subsection{Linear Algebra.}

We define $B_c$ as the $c\times m$ matrix over $\mathbb{F}_2$, or modulo 2, where $c$ is the number of open crossings and $m$ is the number of edges. In other words, the rows are the crossings in the graph, while the columns are the edges. A one is placed in each cell where the corresponding edge is part of the corresponding crossing, while zeros are placed in all other cells. Next, we define $B_v$ as the $v\times m$ matrix over $\mathbb{F}_2$, where $v$ is the number of vertices and $m$ is the number of edges. As in $B_c$, a one is placed in each cell where a vertex and an edge meet, with zeros elsewhere. Stack them as
\[
B=
\begin{pmatrix}
B_v\\
B_c
\end{pmatrix},
\]
so that every edge meeting is recorded in one matrix. For two distinct edges $e$ and $f$,
\[
(B^{\mathsf T}B)_{ef}=1.
\]
This means that, in this new $m\times m$ matrix over $\mathbb{F}_2$, every cell aside from the diagonal is guaranteed to be 1, since it represents the single meeting between two distinct edges. Thus,
\[
B^{\mathsf T}B=J-I+D,
\]
where $J$ is the matrix of all ones and $D$ is the diagonal matrix with $D_{ii}$ equal to the number of crossings on $e_i$ modulo 2. The number of crossings on $e_i=uv$ is $m+1-d(u)-d(v)$, since all edges other than $e_i$ and those incident to it must cross it. Therefore,
\[
D_{ii}\equiv m+1-d(u)-d(v)\pmod 2.
\]
The natural open direction to investigate from here is the ranks of these matrices, since
\[
\operatorname{rank}(J-I+D)
=
\operatorname{rank}(B^{\mathsf T}B)
\le
\operatorname{rank}(B).
\]
Thus, the question becomes: can the matrix $J-I+D$ force rank constraints that restrict how many edges a thrackle can have?

One limitation, however, is that this alone is not enough to determine whether a graph is thrackleable. Even a graph such as $C_4$ can be modeled in this way as though it were thrackleable when it is not. To make this method viable, one would need to add extra information, such as crossing order, to determine whether a graph described by the matrix is thrackleable.

\subsection{Dumbbell Exclusions.}

Dumbbells are extremely powerful constructs, and excluding them provides one of the most important ways to continue with this problem. One general way in which excluding dumbbells could improve the bound is by using Fulek and Pach's inequality. By excluding larger non-thrackleable dumbbells, one can increase the values of $c$ and $\ell$, thus creating a better bound, but only if all dumbbells in the required family of $c$ and $\ell$ values are excluded. Even more promising, though, is discharging. If $DB(6,8,-1)$ were also proved non-thrackleable, and if one could prove that excluding this dumbbell forces $h(F)=0$ for every 8-face $F$, the resulting discharging calculation could provide the hypothetical new bound
\[
\frac{15}{11}.
\]
Further dumbbell exclusions can improve the bound through discharging by allowing even stronger local restrictions to be proved. Another reason that checking the thrackleability of dumbbells is promising is that, if one is proved to be thrackleable, Conway's conjecture is disproved. This is because, for a crossing-free plane embedding of an abstract dumbbell, Euler's formula is
\[
n-m+f=2,
\]
which counts the contained faces and the outer face. If there are two contained faces, as in a dumbbell, then $f=3$, so
\[
n-m+3=2
\quad\Longrightarrow\quad
m=n+1.
\]
Since
\[
n+1>n
\]
by substitution,
\[
m>n
\]
so the conjecture would be disproved.

\subsection{Dumbbell Monotonicity.}

For every graph $G$, if a minor $H$ of it is nonplanar, then $G$ is nonplanar. Fulek and Pach use this principle to determine whether certain abstract graphs are thrackleable, based on a series of graph augmentations that produce an augmented graph $G'$, which is checked for planarity. The idea behind this open direction lies here. Consider specified dumbbells \(H\) and \(G\) such that \(H\preccurlyeq G\). Suppose that, for a small dumbbell $H$, every possible $G'(H,\Pi_H)$ for every complete $\Pi_H$ is tested and found to be nonplanar. Now, for the larger dumbbell $G$, the idea is that if one could prove that, for every complete \(\Pi_G\), the graph \(G'(G,\Pi_G)\) has at least one minor equal to some $G'(H,\Pi_H)$ of the smaller dumbbell, then the larger dumbbell could be proved non-thrackleable without checking every $G'(G,\Pi_G)$. This leaves the following open direction: prove that, for every complete $\Pi_G$, there exists a complete $\Pi_H$ such that $G'(H,\Pi_H)\preccurlyeq G'(G,\Pi_G)$.

\section*{Acknowledgments}

I would like to thank Hung Le for his mentorship throughout this project. He helped guide the direction of the investigation, suggested relevant questions and sources, and provided feedback on the developing work. This research is funded by the NSF CAREER award CCF-2237288 and the NSF grant CCF-2517033.

\section*{Declaration on Generative AI}

Generative AI, specifically ChatGPT 5.6 sol and ChatGPT 6 Astra, were used in a variety of areas while working on this paper. In particular, AI was used in the development of the algorithm, programming, proof checking, verification of the successful run, and accuracy checking of this paper. Starting with the algorithm development, ChatGPT was used to find the speedups of checking a strategically chosen later edge and reflecting the $C_6$ cores to reduce the search from eight cores to four, with both strategies verified by me and all other strategies created by me. For the programming, ChatGPT was used to implement our algorithm based on pseudocode I created, and for auditing the computation, while I also reviewed the program for errors. Aside from the program, I created the mathematical proof of Lemma~\ref{lem:four-half-edges}, with ChatGPT's help limited to advising me about cases omitted from my proof and, at times, suggesting possible directions for filling in those missing cases. For Theorem~\ref{thm:main-bound}, while I am fully responsible for Case 1, ChatGPT pointed out the other two cases I had missed and suggested possible directions, but I was the one who put the mathematics together to complete the proof. Finally, while ChatGPT was an important accuracy checker in the creation of this paper, I wrote the entire paper in my own words, aside from the figures, which I had ChatGPT create after explicitly describing exactly what I wanted in each figure, and a final grammatical check.

\printbibliography

\end{document}